\documentclass[11pt, a4paper, twoside,leqno]{amsart}
\usepackage[centering, totalwidth = 420pt, totalheight = 610pt]{geometry}
\usepackage{amssymb, amsmath, amsthm, enumerate, microtype, stmaryrd, url,mathrsfs}
\usepackage[latin1]{inputenc}
\usepackage[dvips, arrow, matrix, tips, curve]{xy}
\usepackage{color}
\definecolor{darkgreen}{rgb}{0,0.45,0}
\usepackage[pagebackref,colorlinks,citecolor=darkgreen,linkcolor=darkgreen]{hyperref}
\SelectTips{cm}{10}

\DeclareMathOperator{\colim}{colim}
 
\DeclareMathOperator{\im}{Im}
\newcommand{\cat}[1]{\mathbf{#1}}
\newcommand{\op}{\mathrm{op}}

\newcommand{\thg}{{\mathord{\text{--}}}}

\newcommand{\spn}[1]{{\left<{#1}\right>}}

\newcommand{\defeq}{\mathrel{\mathop:}=}

\newcommand{\cd}[2][]{\vcenter{\hbox{\xymatrix#1{#2}}}}

\renewcommand{\phi}{\varphi}
\newcommand{\A}{{\mathscr A}}
\newcommand{\B}{{\mathscr B}}
\newcommand{\C}{{\mathscr C}}
\newcommand{\D}{{\mathscr D}}
\newcommand{\E}{{\mathscr E}}
\newcommand{\F}{{\mathscr F}}

\newcommand{\I}{{\mathscr I}}

\newcommand{\K}{{\mathscr K}}
\renewcommand{\L}{{\mathscr L}}
\newcommand{\M}{{\mathscr M}}

\renewcommand{\P}{{\mathscr P}}
\newcommand{\Q}{{\mathscr Q}}

\newcommand{\Ss}{{\mathscr S}}
\newcommand{\T}{{\mathscr T}}

\newcommand{\V}{{\mathscr V}}

\newcommand{\xtor}[1]{\cdl[@1]{{} \ar[r]|-{\object@{|}}^{#1} & {}}}

\makeatletter

\def\hookleftarrowfill@{\arrowfill@\leftarrow\relbar{\relbar\joinrel\rhook}}
\def\twoheadleftarrowfill@{\arrowfill@\twoheadleftarrow\relbar\relbar}
\def\leftbararrowfill@{\arrowdoublefill@{\leftarrow\mkern-5mu}\relbar\mapstochar\relbar\relbar}
\def\Leftbararrowfill@{\arrowdoublefill@{\Leftarrow\mkern-2mu}\Relbar\Mapstochar\Relbar\Relbar}
\def\leftringarrowfill@{\arrowdoublefill@{\leftarrow\mkern-3mu}\relbar{\mkern-3mu\circ\mkern-2mu}\relbar\relbar}
\def\lefttriarrowfill@{\arrowfill@{\mathrel\triangleleft\mkern0.5mu\joinrel\relbar}\relbar\relbar}
\def\Lefttriarrowfill@{\arrowfill@{\mathrel\triangleleft\mkern1mu\joinrel\Relbar}\Relbar\Relbar}

\def\hookrightarrowfill@{\arrowfill@{\lhook\joinrel\relbar}\relbar\rightarrow}
\def\twoheadrightarrowfill@{\arrowfill@\relbar\relbar\twoheadrightarrow}
\def\rightbararrowfill@{\arrowdoublefill@{\relbar\mkern-0.5mu}\relbar\mapstochar\relbar\rightarrow}
\def\Rightbararrowfill@{\arrowdoublefill@{\Relbar\mkern-2mu}\Relbar\Mapstochar\Relbar\Rightarrow}
\def\rightringarrowfill@{\arrowdoublefill@\relbar\relbar{\mkern-2mu\circ\mkern-3mu}\relbar{\mkern-3mu\rightarrow}}
\def\righttriarrowfill@{\arrowfill@\relbar\relbar{\relbar\joinrel\mkern0.5mu\mathrel\triangleright}}
\def\Righttriarrowfill@{\arrowfill@\Relbar\Relbar{\Relbar\joinrel\mkern1mu\mathrel\triangleright}}

\def\leftrightarrowfill@{\arrowfill@\leftarrow\relbar\rightarrow}
\def\mapstofill@{\arrowfill@{\mapstochar\relbar}\relbar\rightarrow}

\renewcommand*\xleftarrow[2][]{\ext@arrow 20{20}0\leftarrowfill@{#1}{#2}}
\providecommand*\xLeftarrow[2][]{\ext@arrow 60{22}0{\Leftarrowfill@}{#1}{#2}}
\providecommand*\xhookleftarrow[2][]{\ext@arrow 10{20}0\hookleftarrowfill@{#1}{#2}}
\providecommand*\xtwoheadleftarrow[2][]{\ext@arrow 60{20}0\twoheadleftarrowfill@{#1}{#2}}
\providecommand*\xleftbararrow[2][]{\ext@arrow 10{22}0\leftbararrowfill@{#1}{#2}}
\providecommand*\xLeftbararrow[2][]{\ext@arrow 50{24}0\Leftbararrowfill@{#1}{#2}}
\providecommand*\xleftringarrow[2][]{\ext@arrow 10{26}0\leftringarrowfill@{#1}{#2}}
\providecommand*\xlefttriarrow[2][]{\ext@arrow 80{24}0\lefttriarrowfill@{#1}{#2}}
\providecommand*\xLefttriarrow[2][]{\ext@arrow 80{24}0\Lefttriarrowfill@{#1}{#2}}

\renewcommand*\xrightarrow[2][]{\ext@arrow 01{20}0\rightarrowfill@{#1}{#2}}
\providecommand*\xRightarrow[2][]{\ext@arrow 04{22}0{\Rightarrowfill@}{#1}{#2}}
\providecommand*\xhookrightarrow[2][]{\ext@arrow 00{20}0\hookrightarrowfill@{#1}{#2}}
\providecommand*\xtwoheadrightarrow[2][]{\ext@arrow 03{20}0\twoheadrightarrowfill@{#1}{#2}}
\providecommand*\xrightbararrow[2][]{\ext@arrow 01{22}0\rightbararrowfill@{#1}{#2}}
\providecommand*\xRightbararrow[2][]{\ext@arrow 04{24}0\Rightbararrowfill@{#1}{#2}}
\providecommand*\xrightringarrow[2][]{\ext@arrow 01{26}0\rightringarrowfill@{#1}{#2}}
\providecommand*\xrighttriarrow[2][]{\ext@arrow 07{24}0\righttriarrowfill@{#1}{#2}}
\providecommand*\xRighttriarrow[2][]{\ext@arrow 07{24}0\Righttriarrowfill@{#1}{#2}}

\providecommand*\xmapsto[2][]{\ext@arrow 01{20}0\mapstofill@{#1}{#2}}
\providecommand*\xleftrightarrow[2][]{\ext@arrow 10{22}0\leftrightarrowfill@{#1}{#2}}
\providecommand*\xLeftrightarrow[2][]{\ext@arrow 10{27}0{\Leftrightarrowfill@}{#1}{#2}}

\makeatother

\newcommand{\twocong}[2][0.5]{\ar@{}[#2] \save ?(#1)*{\cong}\restore}
\newcommand{\twoeq}[2][0.5]{\ar@{}[#2] \save ?(#1)*{=}\restore}
\newcommand{\rtwocell}[3][0.5]{\ar@{}[#2] \ar@{=>}?(#1)+/l 0.2cm/;?(#1)+/r 0.2cm/^{#3}}
\newcommand{\ltwocell}[3][0.5]{\ar@{}[#2] \ar@{=>}?(#1)+/r 0.2cm/;?(#1)+/l 0.2cm/^{#3}}
\newcommand{\ltwocello}[3][0.5]{\ar@{}[#2] \ar@{=>}?(#1)+/r 0.2cm/;?(#1)+/l 0.2cm/_{#3}}
\newcommand{\dtwocell}[3][0.5]{\ar@{}[#2] \ar@{=>}?(#1)+/u  0.2cm/;?(#1)+/d 0.2cm/^{#3}}
\newcommand{\dltwocell}[3][0.5]{\ar@{}[#2] \ar@{=>}?(#1)+/ur  0.2cm/;?(#1)+/dl 0.2cm/^{#3}}
\newcommand{\drtwocell}[3][0.5]{\ar@{}[#2] \ar@{=>}?(#1)+/ul  0.2cm/;?(#1)+/dr 0.2cm/^{#3}}
\newcommand{\dthreecell}[3][0.5]{\ar@{}[#2] \ar@3{->}?(#1)+/u  0.2cm/;?(#1)+/d 0.2cm/^{#3}}
\newcommand{\utwocell}[3][0.5]{\ar@{}[#2] \ar@{=>}?(#1)+/d 0.2cm/;?(#1)+/u 0.2cm/_{#3}}
\newcommand{\dtwocelltarg}[3][0.5]{\ar@{}#2 \ar@{=>}?(#1)+/u  0.2cm/;?(#1)+/d 0.2cm/^{#3}}
\newcommand{\utwocelltarg}[3][0.5]{\ar@{}#2 \ar@{=>}?(#1)+/d  0.2cm/;?(#1)+/u 0.2cm/_{#3}}

\newdir{(}{{}*!<0em,-.14em>-\cir<.14em>{l^r}}
\newdir{ (}{{}*!/-5pt/\dir{(}}
\newdir{ >}{{}*!/-5pt/\dir{>}}

\theoremstyle{definition}

\swapnumbers
\theoremstyle{plain}
\newtheorem{Thm}[subsection]{Theorem}
\newtheorem{Prop}[subsection]{Proposition}
\newtheorem{Cor}[subsection]{Corollary}
\newtheorem{Lemma}[subsection]{Lemma}

\numberwithin{equation}{section}

\theoremstyle{definition}

\theoremstyle{remark}

\newtheorem{Rk}[subsection]{Remark}

\newcommand{\Lan}{\mathrm{Lan}}
\newcommand{\Ran}{\mathrm{Ran}}

\newcommand{\plex}{\P_l}
\newcommand{\submonad}[1]{{{#1}_l}}
\newcommand{\sat}[1]{{{#1}^\ast}}
\newcommand{\app}[2]{{{#1}_l #2 }}

\newcommand{\appr}[3]{{{#1}_{#2} #3 }}

\begin{document}
 \leftmargini=2em
\title{Lex colimits}
\author{Richard Garner}
\address{Department of Computing, Macquarie University, North Ryde, NSW 2109, Australia}
\email{richard.garner@mq.edu.au}
\author{Stephen Lack}
\address{Department of Mathematics, Macquarie University, North Ryde, NSW 2109, Australia}
\email{steve.lack@mq.edu.au} 
\date{\today}
\thanks{Both authors acknowledge the support of the Australian Research Council and DETYA}
\begin{abstract}
Many kinds of categorical structure require the existence of finite limits, of
colimits of some specified type, and of ``exactness'' conditions between the finite
limits and the specified colimits. Some examples are the notions of regular, or
Barr-exact, or lextensive, or coherent, or adhesive category. We introduce a
general notion of exactness, of which each of the structures listed above, and
others besides, are particular instances. The notion can be understood as a
form of cocompleteness ``in the lex world''---more precisely, in the
$2$-category of finitely complete categories and finite-limit preserving
functors.
\end{abstract}

 \maketitle

\section{Introduction} Amongst the range of structures which it has been found
mathematically useful to impose upon a category, we find a number which share
the following common form. One requires the provision of finite limits; the
provision of colimits of some specified type; and the validation of certain
compatibilities between the finite limits and the specified colimits. For
example, in asking that a category be \emph{lextensive}, or \emph{regular}, or
\emph{Barr-exact}, or \emph{coherent}, or \emph{adhesive}, we are asking for
structure of this form, where the colimits in question comprise the finite
coproducts, or the coequalisers of kernel-pairs, or the coequalisers of
equivalence relations, or the coequalisers of kernel-pairs and the finite
unions of subobjects, or the pushouts along monomorphisms. Though the precise
nature of the limit-colimit compatibilities required varies from case to case,
it is understood that these too share a common form---roughly speaking, they
are just those compatibilities between the finite limits and the specified
colimits which hold in the category of sets; more generally, in any
Grothendieck topos; more generally still, in any $\infty$-pretopos.

The purpose of this paper is to describe a body of results which explains these
similarities of form, by exhibiting each of the structures listed above as
particular instances of a common notion: this notion being one of
``cocompleteness in the lex world''. Let us say a few words about what we mean
by this. The term ``lex'' is here used with the meaning of ``preserving finite
limits''. The etymology of this usage is that, originally, an
additive functor was called ``left exact'' if it preserved exact sequences on
the left; equivalently, if it preserved kernels; equivalently, if it preserved
all finite limits. Then ``left exact'' was abbreviated to ``lex'' and finally
this came to be used to refer to the preservation of finite limits even in the
non-additive context. There is a $2$-category $\cat{LEX}$ comprising the
finitely complete categories, the left exact functors and the natural
transformations between them, and in working in this $2$-category, we may
consider that we are working ``in the lex world''. Thus in speaking of
``cocompleteness in the lex world'', we intend to refer to a notion of
cocompleteness internal to this $2$-category $\cat{LEX}$.

Let us now expand on how such a notion permits a uniform description of each of
the structures listed above. There are two aspects to this. On the one hand, we
must be able to express the kinds of cocompleteness appearing in our examples.
On the other, we must be able to capture the corresponding limit-colimit
compatibilities.

Regarding the first of these, we introduce the following concepts. By a
\emph{class of weights for lex colimits}, we mean a collection $\Phi$ of
functors $\phi \colon \K^\op \to \cat{Set}$, with each $\K$ small and finitely
complete; and by saying that a finitely complete category $\C$ is
\emph{$\Phi$-lex-cocomplete}, we mean that, for every $\phi \colon \K^\op \to
\cat{Set}$ in $\Phi$ and every finite-limit preserving $D \colon \K \to \C$,
the weighted colimit $\phi \star D$ exists in $\C$. For instance, if
$\Phi_\mathrm{ex}$ consists of the single functor $\phi \colon \K^\op \to
\cat{Set}$, where $\K$ is the free category with finite limits generated by an
equivalence relation $(s, t) \colon R \rightarrowtail A \times A$, and where
$\phi$ is the coequaliser in $[\K^\op, \cat{Set}]$ of the maps $\K(\thg, s)$
and $\K(\thg, t)$, then for a finitely complete category $\C$ to be
$\Phi_\mathrm{ex}$-lex-cocomplete is for it to admit coequalisers of
equivalence relations. Similarly, if $\Phi_\mathrm{reg}$ consists of the single
functor $\phi \colon \K^\op \to \cat{Set}$, where $\K$ is the free category
with finite limits on an arrow $f \colon X \to Y$, and where $\phi$ is the
coequaliser in $[\K^\op, \cat{Set}]$ of the kernel pair of $\K(\thg, f)$, then
for a finitely complete category $\C$ to be $\Phi_\mathrm{reg}$-lex-cocomplete
is for it to admit coequalisers of kernel-pairs. In a similar way, we may
express the having of finite coproducts, or of unions of subobjects, or of
pushouts along monomorphisms, as notions of $\Phi$-lex-cocompleteness for
suitable classes $\Phi$.

Our second problem is that of determining, for a given class $\Phi$, the
appropriate compatibilities to be imposed between finite limits and
$\Phi$-lex-colimits. In anticipation of a successful resolution to this, we
reserve the term \emph{$\Phi$-exact} for a finitely complete and
$\Phi$-lex-cocomplete category satisfying these---as yet
undetermined---compatibilities. In due course, we will imbue this term with
meaning in such a way as to capture perfectly our examples: so that, for
instance, a category is $\Phi_\mathrm{reg}$-exact just when it is regular, or
$\Phi_\mathrm{ex}$-exact just when it is Barr-exact. It will turn out that
there are several ways of characterising the notion of $\Phi$-exactness. One of
these is that a \emph{small} $\C$ is $\Phi$-exact just when it admits a full
embedding into a Grothendieck topos via a functor preserving finite limits and
$\Phi$-lex-colimits: this captures the idea, stated above, that the
limit-colimit compatibilities we impose should be just those that obtain in any
Grothendieck topos. A second characterisation of $\Phi$-exactness, valid for
categories of any size, is given in terms of the \emph{$\Phi$-exact completion}
of a finitely complete category $\C$. This will turn out to be the value at $\C$ of a left biadjoint to the forgetful $2$-functor from the $2$-category of $\Phi$-exact categories 
to the $2$-category of finitely complete categories; it includes, for suitable choices of $\Phi$, the exact completion and the regular completion of a category with finite limits, and may be constructed as the full subcategory $\app \Phi \C$
of $[\C^\op, \cat{Set}]$ obtained by closing the representables under finite
limits and $\Phi$-lex-colimits. The second characterisation promised above is now that 
the finitely complete $\C$ is $\Phi$-exact just
when the restricted Yoneda embedding $\C \to \app \Phi \C$ admits a
finite-limit preserving left adjoint. This means that $\C$ is reflective in $\app
\Phi \C$ via a finite-limit-preserving reflector, and so as lex-cocomplete as
$\app \Phi \C$ is; in particular, $\Phi$-lex-cocomplete. Moreover, the same
compatibilities between finite limits and $\Phi$-lex-colimits as are affirmed
in $\cat{Set}$ must be also affirmed in $\app \Phi \C$---since these limits and
colimits are pointwise---and so also in $\C$. Thus this characterisation of
$\Phi$-exactness also accords with our motivating description.

Yet neither of the two \emph{characterisations} of $\Phi$-exactness given above
are satisfactory as a \emph{definition} of it: for whilst justifiable by their
describing correctly the examples we have in mind, they fail to capture the
essence of what $\Phi$-exactness is. As anticipated above, this essence resides
in the claim that $\Phi$-exactness is a transposition ``into the lex world'' of
the notion of cocompleteness with respect to a class of colimits. The force of
this claim is most easily appreciated if we adopt the perspective of monad
theory; in preparation for which, we first recast the standard notions of
cocompleteness in these terms.

When we say that  a category is cocomplete, we are asserting a property, but
this property can be made into a structure: that of being equipped with a
choice of colimits. This structure is algebraic, in the sense that there is a
pseudomonad $\P$ on $\cat{CAT}$, the $2$-category of locally small categories,
whose pseudoalgebras are categories equipped with such colimit structure; the
value of $\P$ at a category $\C$ being given by the closure of the
representables in $[\C^\op, \cat{Set}]$ under small colimits. This same
perspective applies also to notions of partial cocompleteness. For any class of
weights $\Phi$---now meaning simply a collection of presheaves with small
domain---we may again regard the property of being $\Phi$-cocomplete, that is,
of admitting $\phi$-weighted colimits for all $\phi \in \Phi$, as algebraic
structure: we have a pseudomonad $\Phi$ on $\cat{CAT}$ whose pseudoalgebras are
categories equipped with $\Phi$-colimits. We could again describe this
pseudomonad directly---its value at a category $\C$ being the closure of the
representables in $[\C^\op, \cat{Set}]$ under $\Phi$-colimits---but more
pertinently, could also derive it from the pseudomonad $\P$: it is the smallest
full submonad $\Q$ of $\P$ such that $\phi \in \Q \K$ for all $\phi \colon
\K^\op \to \cat{Set}$ in the class $\Phi$. In fact, we may recast even the
definition of a class of weights solely in terms of $\P$; it is given by the
specification, for each small category $\K$, of a full subcategory of $\P\K$.
In other words, once we have the pseudomonad $\P$, representing a notion of
cocompleteness, the corresponding notion of partial cocompleteness may be
derived by a purely formal $2$-categorical process.

This last observation allows us to give form to our claim that $\Phi$-exactness
constitutes a transposition ``into the lex world'' of the standard notion of
cocompleteness with respect to a class of weights. We will exhibit a
pseudomonad $\plex$ on the $2$-category $\cat{LEX}$ which
represents a notion of small-exactness; its pseudoalgebras are the
$\infty$-pretoposes. This $\plex$ is in fact nothing other than the restriction
and corestriction of $\P$ from $\cat{CAT}$ to $\cat{LEX}$, and so represents an
entirely canonical notion of ``cocompleteness in the lex world''. Now applying
the formal $2$-categorical process described above, we obtain the corresponding
notion of ``partial cocompleteness in the lex world'': and this will constitute
our definition of $\Phi$-exactness. The universal property of the $\Phi$-exact completion mentioned above is then an immediate consequence.

(Let us remark here that our approach is related to, but different from, the
work on \emph{Yoneda structures} described in~\cite{Street1978Yoneda}. There,
too, the authors consider an operation $\P$---this being one part of the
definition of a Yoneda structure---which, in the case of their Example~7.3,
resides on $\cat{LEX}$. As for any Yoneda structure, one may define notions of
cocompleteness or partial cocompleteness with respect to this $\P$; however,
the notions so arising are not the same as our small-exactness or
$\Phi$-exactness. For instance, as was stated above, to be small-exact in our
sense is to be an $\infty$-pretopos. The corresponding notion in the framework
of~\cite{Street1978Yoneda} is the stronger one of being
\emph{lex-total}~\cite{Street1981Notions}.)

As we have already said, the notion of $\Phi$-exactness captures perfectly our
motivating examples: however, it also allows us to move beyond those examples.
For instance, we shall see that when $\Phi$ comprises the weights for finite
unions, a finitely complete category is $\Phi$-exact just when it admits finite unions which are both  stable under
pullback and \emph{effective}---calculated as the pushout over the
intersection. Similarly, we shall see that if $\Phi$ comprises the weights for filtered colimits, then a finitely complete category is $\Phi$-exact just when it has filtered colimits which commute with finite limits; and finally, that if $\Phi$ comprises the weight for reflexive
coequalisers, then a finitely complete category is $\Phi$-exact just when it is Barr-exact and
the free equivalence relation on each reflexive relation exists, and is calculated in the same way
as in $\cat{Set}$. Let us be clear that the properties just listed are
neither new nor unexpected: what \emph{is} new is the understanding that the
structures they define stand on an equal footing with our motivating ones.

One further pleasant aspect of the theory we develop is that it works just as well for
enriched as for ordinary categories. In the $\V$-categorical setting, the
notion of $\Phi$-exactness involves inheriting limit-colimit compatibilities
from $\V$, rather than from $\cat{Set}$, which on a concrete level may cause
the theory to look quite different, for different choices of $\V$. Our formal
development will be given in the enriched context from the outset; when it
comes to examples, however, we shall limit the scope of this paper to the
motivating case $\V = \cat{Set}$, leaving applications over other bases for
future investigation. Let us at least remark that amongst these applications
are the case $\V = \cat{Ab}$---which should allow us to capture the various
exactness notions of~\cite{Grothendieck1957Sur-quelques}---and the case $\V =
\cat{Cat}$---which should allow us to provide a clear conceptual basis for
various forms of $2$-categorical exactness
\cite{Bourke2010Codescent,Bourn20102-categories,Carboni1994Modulated,Street1982Two-dimensional}

\looseness=-1
Let us now give a brief overview of the content of this paper. We begin in
Section~\ref{sec:1} by recalling the construction of the pseudomonad $\P$ on
$\V\text-\cat{CAT}$, and describing its lifting to a pseudomonad $\plex$ on
$\V\text-\cat{LEX}$. In Section~\ref{sec:phiexactness}, we go on to consider
full submonads of $\plex$, so arriving at our definition of $\Phi$-exactness.
Then in Section~\ref{sec:embedding}, we give the embedding result
described above---which, in the enriched context, states that a small,
$\Phi$-lex-cocomplete $\C$ is $\Phi$-exact just when it admits an embedding in
a ``$\V$-topos''; that is, a $\V$-category reflective in a presheaf category by
a finite-limit-preserving reflector. The more involved parts of the proof are
deferred to Section 7 and an Appendix. In Section~\ref{sec:3}, we break off
from the general theory in order to give a body of examples; as remarked above,
these examples will be concerned solely with the case where $\V = \cat{Set}$.
In Section~\ref{sec:4}, we show that, again in the case $\V = \cat{Set}$, it is
possible to give a concrete characterisation of $\Phi$-exactness for an
\emph{arbitrary} $\Phi$, in terms of Anders Kock's notion of \emph{postulated
colimit}~\cite{Kock1989Postulated}. Then in Section~\ref{sec:relative}, we
resume our development of the general theory, describing the construction of
\emph{relative completions}: that is, of the free $\Psi$-exact category on a
$\Phi$-exact one, for suitable classes of weights $\Phi$ and $\Psi$. Finally,
an Appendix proves some necessary technical results concerning localisations of
locally presentable categories.

\textbf{Acknowledgements}. We should say some words about the prehistory of
this project. Max Kelly observed that the regular completion and the exact
completion of a category with finite limits~\cite{Hu1996A-note} can be computed
as full subcategories of the presheaf category. He proposed that this should be
explained by the fact that these were ``free cocompletions in the lex world'',
and he observed that the existence of coequalisers of equivalence relations
could be seen as $\Phi$-lex-cocompleteness, in essentially the same sense
considered here, for a suitably chosen class of weights $\Phi$. He planned to
study this with the second-named author, with a view to explaining the
construction of~\cite{Hu1996A-note}, but other things intervened and the project
never progressed very far. Some years later, the first-named author encountered
a remark asserting the existence of the project in the second
author's~\cite{Lack1999A-note}; observing that this was, in fact, the only trace
of its existence, and perceiving how some basic aspects of the theory should
go, he made contact with the second author and work on the project was begun
anew, resulting in the present article. Let us observe that Kelly's original
goal is fulfilled by our Corollary~\ref{cor:biadjoint}.

\section{Cocomplete and small-exact categories}\label{sec:1}
In this section, we recall the construction of the free cocompletion
pseudomonad $\P$ on $\cat{CAT}$, and give various characterisations of its
pseudoalgebras; they are, of course, the cocomplete categories. We then
describe the lifting of $\P$ to a pseudomonad $\plex$ on the $2$-category of
finitely complete categories, and give various analogous characterisations of
the $\plex$-pseudoalgebras; we call a category bearing such algebra structure
\emph{small-exact}.

As we have already said, we shall work from the outset in the context of the
enriched category theory of~\cite{Kelly1982Basic}. Thus we fix a symmetric
monoidal closed category $\V$, and henceforth write category to mean
$\V$-category, functor to mean $\V$-functor, and so on; in particular, when we
speak of limits and colimits, we mean the \emph{weighted} (there called
\emph{indexed}) limits and colimits of~\cite[Chapter 3]{Kelly1982Basic}. We
shall also assume that $\V$ is \emph{locally finitely presentable as a closed
category} in the sense of~\cite{Kelly1982Structures}; which is to say that its
underlying category $\V_0$ is locally finitely presentable---so in particular,
complete and cocomplete---and that the finitely presentable objects are closed
under the monoidal structure. This will be necessary later to ensure that we
have a good notion of finite limit in our enriched setting.

It will also do us well to be clear on some foundational matters. We assume the
existence of an inaccessible cardinal $\infty$, and call a set \emph{small} if
of cardinality $< \infty$, and a $\V$-category \emph{small} if having only a
small set of isomorphism-classes of objects. As usual, a set or category which
is not small is called \emph{large}; we may sometimes refer to a large set as a
\emph{class}. Now when we say that a category is complete or cocomplete, we
really mean to say that it is small-complete or small-cocomplete, in the sense
of having limits or colimits indexed by weights with small domain. $\cat{Set}$
is the category of \emph{small} sets; which, means in particular that the case
$\V = \cat{Set}$ of our general notions will be concerned with \emph{locally
small} ordinary categories.

Let $\cat{CAT}$ denote the $2$-category of (possibly large) categories,
functors and natural transformations; by which we mean, of course,
$\V$-categories, $\V$-functors and $\V$-natural transformations, so that our
$\cat{CAT}$ is what might otherwise be denoted $\V$-$\cat{CAT}$. Let
$\cat{COCTS}$ denote the locally full sub-$2$-category of $\cat{CAT}$
comprising the cocomplete categories and cocontinuous functors. There is a
forgetful $2$-functor $\cat{COCTS} \to \cat{CAT}$, and by a \emph{free
cocompletion} of a category $\C$, we mean a bireflection of it along this
$2$-functor. This amounts to the provision of a category $\P \C$ and functor $Y
\colon \C \to \P \C$ with the property that, for each cocomplete $\D$, the
functor
\begin{equation}\label{eq:freecouniv}
    \cat{COCTS}(\P\C, \D) \to \cat{CAT}(\C, \D)
\end{equation}
induced by composition with $Y$ is an equivalence of categories. As is
explained in~\cite[\S 5.7]{Kelly1982Basic}, every category admits a free
cocompletion; it may be described as follows. We declare a presheaf $\phi
\colon \C^\op \to \V$ to be \emph{small} when it is the left Kan extension of
its restriction to some small full subcategory $\L$ of $ \C$, and define the
category $\P \C$ to have as objects, the small presheaves on $\C$, and
hom-objects $\P \C(\phi, \psi)$ given by the usual end formula $\int_{X \in
\C}[\phi X, \psi X]$. Note that this large end, which \emph{a priori} need not
exist in $\V$, may be calculated as the small end $\int_{X \in \L} [\phi JX,
\psi JX]$ for any $J \colon \L \hookrightarrow \C$ witnessing the smallness of
$\phi$. The functor $Y \colon \C \to \P \C$ takes $X \in \C$ to the
representable presheaf $\C(\thg, X)$. Observe that when $\C$ is small, every
presheaf on $\C$ is small, so that $\P \C = [\C^\op, \V]$ and $Y$ is the Yoneda
embedding. The following is now (a special case of) Theorem~5.35
of~\cite{Kelly1982Basic}.
\begin{Prop}
For every category $\C$, the category $\P \C$ of small presheaves on $\C$,
together with its restricted Yoneda embedding $Y \colon \C \to \P \C$, is a
free cocompletion of $\C$.
\end{Prop}
The equivalence inverse of~\eqref{eq:freecouniv} takes a functor $F \colon \C
\to \D$ with cocomplete domain to the functor $\bar F \colon \P \C \to \D$
which sends $\phi \in \P \C$ to the weighted colimit $\phi \star F$; as before,
this large colimit, which \emph{a priori} need not exist in $\D$, may be
computed as the small colimit $\phi J \star FJ$ for any $J \colon \L
\hookrightarrow \C$ witnessing the smallness of $\phi$. Observe that $\bar F$
is equally well the left Kan extension of $F$ along $Y \colon \C \to \P \C$, by
which we mean the pointwise left Kan extension; in this paper we shall consider
no other kind.

The universal property of free cocompletion induces a pseudomonad structure on
$\P$. The action of $\P$ on morphisms sends a functor $F \colon \C \to \D$ to
the functor $\P \C \to \P \D$ obtained as the cocontinuous extension of $YF
\colon \C \to \P \D$; this is equally well the functor sending $\phi \in \P \C$
to $\Lan_{F^\op}(\phi) \in \P \D$. The unit of the pseudomonad at $\C$ is $Y
\colon \C \to \P \C$ whilst the multiplication $M \colon \P \P \C \to \P \C$ is
the cocontinuous extension of the identity functor $\P \C \to \P \C$; thus
$M(\phi) \cong \phi \star 1_{\P \C}$.

This pseudomonad is of the kind which has been called
\emph{Kock-Z\"oberlein}---see \cite{Kock1995Monads} and the references
therein---but for which we adopt, following~\cite{Kelly1997On-property-like}, the
more descriptive name \emph{lax-idempotent}. The characteristic property of
such pseudomonads is that ``structure is left adjoint to unit''; more
precisely, this means that pseudoalgebra structures on an object correspond
with left adjoint reflections for the unit map at that object. In the case of
$\P$, the admission of such a left adjoint is easily seen to coincide with the
property of being cocomplete; in more detail, we have the following result.
\begin{Prop}\label{prop:charcocomp} For a category $\C$, the following are equivalent:
\begin{enumerate}
\item $\C$ is (small-)cocomplete;
\item For each $D \colon \K \to \C$ and $\phi \in \P \K$, the colimit $\phi
    \star D$ exists in $\C$;
\item For every $\phi \in \P \C$, the colimit $\phi \star 1_\C$ exists in
    $\C$;
\item For each $D \colon \K \to \C$ with $\K$ small, the Kan extension
    $\Lan_Y D \colon \P \K \to \C$ exists;
\item For each $D \colon \K \to \C$, the Kan extension $\Lan_Y D \colon \P
    \K \to \C$ exists;
\item The Kan extension $\Lan_Y (1_\C) \colon \P \C \to \C$ exists;
\item The functor $Y \colon \C \to \P \C$ admits a left adjoint;
\item $\C$ admits a structure of $\P$-pseudoalgebra;
\item $\C$ is reflective in some $\P \D$.
\end{enumerate}
\end{Prop}
\begin{proof}
(1) $\Rightarrow$ (2) by the preceding observations; (2) $\Rightarrow$ (3)
trivially. (1) $\Leftrightarrow$ (4), (2) $\Leftrightarrow$ (5) and (3)
$\Leftrightarrow$ (6) since $\Lan_Y D(\phi) \cong \phi \star D$. (6)
$\Rightarrow$ (7) since $\Lan_Y (1_\C) \colon \P \C \to \C$ is, by basic
properties of Kan extensions, left adjoint to $Y$. (7) $\Leftrightarrow$ (8)
because $\P$ is lax-idempotent. (7) $\Rightarrow$ (9) since $Y$ is fully
faithful, so that if it admits a left adjoint, then $\C$ is reflective in $\P
\C$. (9) $\Rightarrow$ (1) since any category reflective in a cocomplete
category is cocomplete.
\end{proof}
We now consider the interaction between the pseudomonad $\P$ and finite limit
structure on a category. We begin by recalling from~\cite{Kelly1982Structures}
some necessary definitions. A \emph{finite weight} is a functor $\phi \colon
\K^\op \to \V$ such that $\K$ has a finite set of isomorphism-classes of
objects, with each hom-object $\K(X, Y)$ and each $\phi(X)$ being finitely
presentable in $\V$. A weighted limit is called \emph{finite} if its weight is
finite, a category is \emph{finitely complete} if it admits all finite limits,
and a functor between finitely complete categories is \emph{left exact} if it
preserves finite limits; we sometimes write \emph{lex} for \emph{left exact}.
The proof of the following result is now contained in Proposition 4.3 and
Remark 6.6 of~\cite{Day2007Limits}; though in the case $\V = \cat{Set}$, the
result is much older. Since we shall not need the details of the proof in what
follows, we do not recount them here.
\begin{Prop}\label{prop:lift}\hfill
\begin{enumerate}
\item If $\C$ is finitely complete, then so is $\P \C$;
\item If $F \colon \C \to \D$ is left exact, then so is $\P F \colon \P \C
    \to \P \D$;
\item For any finitely complete $\C$, both $Y \colon \C \to \P \C$ and $M \colon \P \P \C \to
    \P \C$ are left exact.
\end{enumerate}
\end{Prop}

As in the Introduction, let us write $\cat{LEX}$ for the locally full
sub-$2$-category of $\cat{CAT}$ comprising the finitely complete categories and
the left exact functors. It follows from Proposition~\ref{prop:lift} that $\P$
restricts and corestricts to a lax-idempotent pseudomonad on $\cat{LEX}$,
which, as in the Introduction, we shall denote by $\plex$. We now wish to
characterise the $\plex$-pseudoalgebras. Since $\plex$ is lax-idempotent, and
all of its unit maps are fully faithful, pseudoalgebra structures on the
finitely complete $\C$ may be identified with left adjoints in $\cat{LEX}$ for
the unit $Y \colon \C \to \plex \C = \P \C$. Such a left adjoint in $\cat{LEX}$
is of course also one in $\cat{CAT}$, and so any $\plex$-pseudoalgebra is
cocomplete. The extra requirement that the left adjoint should be left exact
may be rephrased in a number of ways, by analogy with
Proposition~\ref{prop:charcocomp}.

\begin{Prop}\label{prop:charlexcocomp} For a finitely complete and cocomplete category $\C$, the following are equivalent:
\begin{enumerate}
\item For each lex $D \colon \K \to \C$ with $\K$ small, the functor
    $(\thg) \star D \colon \P \K \to \C$ is also lex;
\item For each lex $D \colon \K \to \C$, the functor $(\thg) \star D \colon
    \P \K \to \C$ is also lex;
\item The functor $(\thg) \star 1_\C \colon \P \C \to \C$ is lex;
\item For each lex $D \colon \K \to \C$ with $\K$ small, the functor
    $\Lan_Y D \colon \P \K \to \C$ is also lex;
\item For each lex $D \colon \K \to \C$, the functor $\Lan_Y D \colon \P \K
    \to \C$ is also lex;
\item The functor $\Lan_Y (1_\C) \colon \P \C \to \C$ is lex;
\item The functor $Y \colon \C \to \P \C$ admits a left exact left adjoint;
\item $\C$ admits a structure of $\plex$-pseudoalgebra;
\item $\C$ is lex-reflective in some $\P \D$ with $\D$ finitely complete.
\end{enumerate}
\end{Prop}
In part (9), a category $\A$ is said to be \emph{lex-reflective} in a category
$\B$ if there is a fully faithful functor $\A \to \B$ which admits a left exact
left adjoint; we may sometimes also say that $\A$ is a \emph{localisation} of
$\B$.
\begin{proof}
The only implications not exactly as before are (1) $\Rightarrow$ (2) and (9)
$\Rightarrow$ (1). For the former, let $D \colon \K \to \C$ be lex; assuming
(1), we must show that $(\thg) \star D \colon \P \K \to \C$ preserves finite
limits. So given $\psi \colon \M \to \V$ a finite weight and $H \colon \M \to
\P \K$, we are to show that $\{\psi, H\} \star D \cong \{\psi ?, H? \star D\}$.
Choose some $J \colon \L \hookrightarrow \K$ which witnesses the smallness of
$HX \in \P\K$ for every $X \in \M$; without loss of generality, we may assume
that $\L$ is closed under finite limits in $\K$, so that $\L$ is finitely
complete and $J$ lex. Write $\bar H \colon \M \to \P \L$ for the functor
sending $X$ to $(HX).J^\op \colon \L^\op \to \V$. Then $H \cong
\Lan_{J^\op}.\bar H$, and $\Lan_{J^\op} \cong \P J$ preserves finite limits by
Proposition~\ref{prop:lift}, whence
\begin{align*}
    \{\psi, H\} \star D &\cong
    \{\psi, \Lan_{J^\op}.\bar H\} \star D \cong
    (\Lan_{J^\op}\{\psi, \bar H\}) \star D \cong
    \{\psi, \bar H\} \star DJ \\ &\cong
    \{\psi?, \bar H ? \star DJ \}\cong
    \{\psi?, \Lan_{J^\op}(\bar H ?) \star D \} \cong
    \{\psi?,  H ? \star D \}
\end{align*}
where in passing from the first to the second line we use (1) applied to the
lex $DJ$ with small domain. This proves that (1) $\Rightarrow$ (2); it remains
only to show that (9) $\Rightarrow$ (1). Let $\D$ be finitely complete, and let
$L \dashv J \colon \C \to \P \D$ exhibit $\C$ as a localisation of $\P \D$. Now
given $\K$ small and $D \colon \K \to \C$ lex, the functor $(\thg) \star D$ may
be calculated to within isomorphism as the composite
\begin{equation*}
    \P \K \xrightarrow{\P D} \P \C \xrightarrow{\P J} \P \P \D \xrightarrow{M} \P \D \xrightarrow{L} \C
\end{equation*}
each of whose constituent parts is lex either by Proposition~\ref{prop:lift} or
by assumption; whence the composite is lex as required.
\end{proof}
We shall call a category satisfying any of the equivalent conditions of this
proposition \emph{small-exact}. In the case where $\V = \cat{Set}$ we can give
a concrete characterisation of the small-exact categories. Recall that an
\emph{$\infty$-pretopos} is a finitely complete and small-cocomplete
$\cat{Set}$-category in which colimits are stable under pullback, coproduct
injections are disjoint, and every equivalence relation is effective.
\begin{Prop}
A finitely complete and small-cocomplete $\cat{Set}$-category is small-exact if
and only if it is an $\infty$-pretopos.
\end{Prop}
\begin{proof}
$\cat{Set}$ is certainly an $\infty$-pretopos; whence also any
$\cat{Set}$-category $\P \D$ where $\D$ is lex, since finite limits and small
colimits in $\P \D$ are computed pointwise. It is moreover easy to show that
the exactness properties of an $\infty$-pretopos will be inherited by any
localisation of it; and so every small-exact $\cat{Set}$-category is an
$\infty$-pretopos by clause (9) of Proposition~\ref{prop:charlexcocomp}. For
the converse, we observe that
 any $\infty$-pretopos satisfies clause (4) of Proposition~\ref{prop:charlexcocomp}---see,
 for instance,~\cite[Corollary 3.3]{Kock1989Postulated}---and so is small-exact.
\end{proof}
Returning to the case of a general $\V$, let us define a \emph{$\V$-topos} to
be any localisation of a presheaf category $[\C^\op, \V]$ on a small $\C$. The
following result can be seen as a ``Giraud theorem''; when $\V = \cat{Set}$, it
recaptures \cite{Giraud1964Analysis}'s characterisation of the Grothendieck
toposes as the $\infty$-pretoposes with a small generating family (bearing in
mind that in an $\infty$-pretopos, the full subcategory spanned by any
generating family is dense).
\begin{Prop}\label{prop:vtopsmallexact}
The finitely complete $\E$ is a $\V$-topos if and only if it is small-exact and
has a small, dense subcategory.
\end{Prop}
\begin{proof}
Suppose first that $\E$ is a localisation of $[\C^\op, \V]$ for some small
$\C$. Certainly $\E$ has a small dense subcategory, given by the full image of
the representables under the reflector $[\C^\op, \V] \to \E$; we must show that
$\E$ is also small-exact. So let $\D$ denote the closure of the representables
in $[\C^\op, \V]$ under finite limits, and let $J \colon \C \to \D$ be the
restricted Yoneda embedding. It is easy to see that $\D$ is again small; and
now the composite of the fully faithful $\Ran_J \colon [\C^\op, \V] \to
[\D^\op, \V]$ with $\E \to [\C^\op, \V]$ manifests $\E$ as lex-reflective in
$[\D^\op, \V]$. Since $\D$ is finitely complete, $\E$ is small-exact by
Proposition~\ref{prop:charlexcocomp}(9).

Conversely, suppose $\E$ is small-exact, with a small dense subcategory $J
\colon \C \to \E$. Upon replacing $\C$ by its finite-limit closure in
$\E$---which will again be small and dense---we may assume that $\C$ is
finitely complete, and $J$ left exact, so that by
Proposition~\ref{prop:charlexcocomp}(4), $\Lan_Y J \colon [\C^\op, \V] \to \E$
is also left exact. But this functor has as right adjoint the singular functor
$\tilde J  \colon \E \to [\C^\op, \V]$, which is fully faithful as $\C$ is
dense; whence $\E$ is lex-reflective in $[\C^\op, \V]$, and so a $\V$-topos as
required.
\end{proof}

\section{$\Phi$-exactness}\label{sec:phiexactness}
In the previous section we considered pseudoalgebras for the pseudomonad
$\plex$ on $\cat{LEX}$. In this section, we consider pseudoalgebras for
suitable full submonads of $\plex$; these will be the $\Phi$-exact categories
which are the primary concern of this paper. Let us begin by defining what we
mean by a full submonad of $\plex$. Suppose that we are given, for each
finitely complete $\C$, a subcategory $\Q\C \subseteq \P \C$ which is full,
replete and closed under finite limits, with these choices being such that:
\begin{enumerate}
\item For all finitely complete $\C$, the map $Y \colon \C \to \P\C$
    factors through $\Q\C$ (i.e., $\Q \C$ contains the representables);
\item For all lex $F \colon \C \to \D$, the functor $\P F \colon \P\C
    \to \P \D$ maps $\Q \C$ into $\Q \D$;
    \item For all finitely complete $\C$, the functor $M \colon \P\P\C \to \P\C$ maps $\Q \Q \C$ into $\Q \C$.
\end{enumerate}
Under these circumstances, it is easy to see that $\Q \C$ is the value at $\C$
of a lax-idempotent pseudomonad $\Q$ on $\cat{LEX}$, and that the inclusions
$ \Q \C \hookrightarrow \P \C$ constitute a pseudomonad morphism $J \colon \Q \to \plex$. We
shall then call $\Q$ a \emph{full submonad} of $\plex$.

In order to generate full submonads of $\plex$, we define, as in the
Introduction, a \emph{class of weights for lex colimits}---more briefly, a
\emph{class of lex-weights}---to be given by a collection $\Phi$ of functors
$\phi \colon \K^\op \to \V$, with each $\K$ small and finitely complete. Before
continuing, let us stress that by a lex-weight, we do not mean a weight that is
lex; there is no requirement that the $\phi$'s should preserve finite limits,
only that they should be presheaves on finitely complete categories. Let us
also introduce some notation: given a class of lex-weights $\Phi$, we write
$\Phi[\K]$ for the full subcategory of $[\K^\op, \V]$ spanned by the functors
in $\Phi$ with domain $\K^\op$.

From each class of lex-weights $\Phi$ we may generate a full submonad of
$\plex$: namely, the smallest full submonad $\Q$ such that $\Phi[\K] \subseteq
\Q \K$ for each small, finitely complete $\K$. We denote this submonad by $\submonad{\Phi}$, and for
each finitely complete $\C$, write
\begin{equation*}
    \C \xrightarrow{W} \app \Phi \C \xrightarrow{J} \P \C
\end{equation*}
for the corresponding factorisation of $Y \colon \C \to \P \C$. From this
submonad $\submonad{\Phi}$ we obtain in turn a new class of lex-weights
$\Phi^\ast$, comprising all those $\phi \colon \K^\op \to \V$ which lie in 
$\app \Phi \K$ for some small, finitely complete $\K$.
We call this class $\sat \Phi$ the \emph{saturation} of $\Phi$,
and say that $\Phi$ is \emph{saturated} if $\Phi = \sat \Phi$. This name is
justified by the (easy) observation that saturation is a closure operator on
classes of lex-weights. Let us warn the reader that the saturation of a class
of lex-weights as we have just defined it is \emph{not} the same as  its
saturation in the sense of~\cite{Albert1988The-closure,Kelly2005Notes}; ours is a
saturation ``in the lex world'', which, as the following result shows, involves
closure under finite limits as well as under the specified colimits.


\begin{Prop}\label{prop:phicchar}
For each finitely complete $\C$, the category $\app \Phi \C$ is the closure of
the representables in $\P \C$ under finite limits and $\Phi$-lex-colimits.
\end{Prop}
By this, we mean that $\app \Phi \C$ is the smallest full, replete subcategory
of $\P\C$ which contains the representables and is closed under finite limits
and $\Phi$-lex-colimits; as in the introduction, a \emph{$\Phi$-lex-colimit} in a finitely complete category $\C$ is a weighted colimit of the form $\phi \star D$ for some $\phi \in \Phi[\K]$ and some lex functor $D \colon \K \to \C$.
\begin{proof}
Before beginning the proof proper, let us observe that we may combine conditions (2) and (3) for a full submonad into the single condition that, for every lex $F \colon \C \to \Q \D$, the composite
\begin{equation}
\label{eq:altform}
 \P \C \xrightarrow{\P F} \P \Q \D \xrightarrow{\P J} \P \P \D \xrightarrow{M} \P \D
\end{equation}
should map $\Q \C$ into $\Q \D$. As the displayed composite sends $\phi \in \P \C$ to the colimit $\phi \star JF$ in $\P \D$, this is equally well to ask that, for every $\phi \in \Q \C$ and lex $F \colon \C \to \Q \D$, the colimit $\phi \star JF$ in $\P \D$ should lie in $\Q \D$; and we may express this by saying that $\Q \D$ is closed in $\P \D$ under $\Q$-lex-colimits.

Now for each finitely complete $\C$, let $\Q \C$ denote the closure of the representables in $\P \C$ under finite limits and $\Phi$-lex-colimits. We first show that the $\Q\C$'s constitute a full submonad $\Q$ of $\plex$. Clearly each $\Q\C$ contains the representables, and is full, replete and finite-limit-closed in $\P \C$, and so by the above discussion, it is enough to show that for each lex $F \colon \C \to \Q \D$, the composite~\eqref{eq:altform} maps $\Q \C$ into $\Q \D$. But for each $F$, the collection of $\phi \in \P \C$ for which~\eqref{eq:altform} lands in $\Q \D$ is easily seen to contain the representables; it is moreover closed under finite limits and $\Phi$-lex-colimits, since~\eqref{eq:altform} preserves them, and $\Q \D$ is closed in $\P \D$ under them, and so must encompass all of $\Q \C$, as desired. 

So we have a full submonad $\Q \subseteq \plex$; moreover, for each small finitely complete $\K$, we have $\Phi[\K] \subseteq \Q \K$, as any $\phi \in \Phi[\K]$ can be expressed as $\phi \star Y$---with $Y \colon \K \to [\K^\op, \V]$ the (lex) Yoneda embedding---and so as a $\Phi$-lex-colimit of representables. Thus $\Q$ is a full submonad of $\plex$ containing $\Phi$, and we now claim that it is the smallest such; in other words, that  $\Q = \submonad \Phi$ as desired.  Thus given any full submonad $\Q' \subseteq \plex$, we must show that $\Q \subseteq \Q'$; for which it will suffice to show that each $\Q' \C$  contains the representables, and is full, replete and closed under finite limits and $\Phi$-lex-colimits in $\P \C$. The only non-trivial clause is closure under $\Phi$-lex-colimits. Since $\Q'$ is a full submonad, we know from the above discussion that each $\Q' \C$ is closed in $\P \C$ under $\Q'$-lex-colimits; and as $\Phi[\K] \subseteq \Q' \K$ for each small, finitely complete $\K$, each $\Q' \C$ is thereby closed in $\P \C$ under $\Phi$-lex-colimits, as required.
\end{proof}
We shall also make use of the following result, which says that the pseudomonad
generated by a class of lex-weights is ``small-accessible'':
\begin{Prop}\label{prop:keylemma}
If $\Phi$ is a class of lex-weights, and $\C$ a finitely complete category,
then every $\phi \in \app \Phi \C$ is of the form $\Lan_{J^\op}(\psi)$ for some
lex $J \colon \L \to \C$ with small domain and some $\psi
\in \app \Phi \L$. In fact, we may always take $J$ to be the inclusion of a
full, replete subcategory, and $\psi$ to be the composite $\phi J$.
\end{Prop}
\begin{proof}
Let $\E$ denote the subcategory of $\app \Phi \C$ spanned by those $\phi$
satisfying the stronger form of the stated condition. Clearly the
representables lie in $\E$. Now suppose that $\psi \colon \K \to \V$ is a
finite weight and $D \colon \K \to \app \Phi \C$ is such that every $DX$ lies
in $\E$; we shall show that $\{\psi, D\}$ does too. If the subcategory $J_X
\colon \L_X \hookrightarrow \C$ witnesses the condition for $DX$, then taking
$J \colon \L \hookrightarrow \C$ to be the lex closure of the union of these
subcategories, we have $\{\psi, D\} \cong \{\psi, \Lan_{J^\op}(D(\thg)J)\}
\cong \Lan_{J^\op} \{\psi, D(\thg)J\}$. By assumption, each $DX.J_X \in \app
\Phi {\L_X}$, whence easily $DX.J \in \app \Phi \L$ and so $\{\psi, D(\thg)J\}
\in \app \Phi \L$, as $\app \Phi \L$ is closed under finite limits in $\P \L$.
Thus $\{\psi, D\} J \cong \{\psi, D(\thg)J\}$ is in $\app \Phi \L$ and
$\Lan_{J^\op}(\{\psi, D\}J) \cong \{\psi, D\}$ as claimed. An entirely similar
argument shows $\E$ is closed under $\Phi$-lex-colimits in $\app \Phi \C$; and
so by the preceding proposition, we have $\E = \app \Phi \C$.
\end{proof}

Given a class of lex-weights $\Phi$, we now give a characterisation of the
$\submonad \Phi$-pseudoalgebras. Since $\submonad \Phi$ is lax-idempotent and
all of its unit maps are fully faithful, the finitely complete $\C$ will admit
$\submonad \Phi$-pseudoalgebra structure just when the unit map $W \colon \C
\to \app \Phi \C$ admits a left adjoint in $\cat{LEX}$. The following two
propositions characterise, firstly, those $\C$ for which a left adjoint to $W$
exists in $\cat{CAT}$, and secondly, those for which such a left adjoint is
left exact. Given a class of lex-weights $\Phi$, we say as in the Introduction
that $\C$ is \emph{$\Phi$-lex-cocomplete} if it is finitely complete, and for
every $\phi \in \Phi[\K]$ and every lex $D \colon \K \to \C$, the colimit $\phi
\star D$ exists in $\C$.
\begin{Prop}\label{prop:charphicocomp}Let $\Phi$ be a class of lex-weights, and
let $\C$ be finitely complete. Then the following are equivalent:
\begin{enumerate}
\item $\C$ is $\sat \Phi$-lex-cocomplete;
\item For each lex $D \colon \K \to \C$ and $\phi \in \app \Phi \K$, the
    colimit $\phi \star D$ exists in $\C$;
\item For every $\phi \in \app \Phi \C$, the colimit $\phi \star 1_\C$
    exists in $\C$;
\item For each lex $D \colon \K \to \C$ with $\K$ small, the Kan extension
    $\Lan_W D \colon \app \Phi \K \to \C$ exists;
\item For each lex $D \colon \K \to \C$, the Kan extension $\Lan_W D \colon
    \app \Phi \K \to \C$ exists;
\item The Kan extension $\Lan_W (1_\C) \colon \app \Phi \C \to \C$ exists;
\item The functor $W \colon \C \to \app \Phi \C$ admits a left adjoint;
\item $\C$ is reflective in some $\app \Phi \D$.
\end{enumerate}
\end{Prop}
The proof of this result is entirely analogous to that of
Proposition~\ref{prop:charcocomp}, though making essential use of
Proposition~\ref{prop:keylemma} for the implication (1) $\Rightarrow$ (2). Let
us remark on a further condition the finitely complete $\C$ may fulfil which is
\emph{not} on the preceding list, by virtue of its being strictly weaker:
namely, the condition of being $\Phi$-lex-cocomplete, as opposed to
$\Phi^\ast$-lex-cocomplete. Whilst the two are equivalent for many important
classes of weights, it need not always be so. For example, in
Section~\ref{subsec:reflexive} below, we shall meet a class of lex-weights
$\Phi_\mathrm{rc}$ such that a finitely complete $\cat{Set}$-category $\C$ is
$\Phi_\mathrm{rc}$-lex-cocomplete just when it admits coequalisers of reflexive
pairs. However, for such a $\C$ to be $\Phi^\ast_\mathrm{rc}$-lex-cocomplete,
it must admit certain additional colimits, related to the construction of the
free equivalence relation on a reflexive relation;  see
Proposition~\ref{prop:reflexivechar}. The reason for the discrepancy is that,
on closing the representables in $\P \C$ under coequalisers of reflexive pairs,
the resultant subcategory need no longer be closed under finite limits; and taking
this closure---as we must do in forming $\Phi_\mathrm{rc} \C$---introduces new
weights for colimits, not constructible from coequalisers of reflexive pairs alone, which
any $\sat \Phi$-lex-cocomplete category must admit. If, however, $\Phi$ is a class of lex-weights with the property that the closure of the
representables in $\P \C$ under $\Phi$-lex-colimits is already closed under
finite limits---as happens in the case $\V = \cat{Set}$ when $\Phi$ is the
class of weights for finite coproducts, or for coequalisers of kernel-pairs, or
for coequalisers of equivalence relations---then $\Phi$-lex-cocompleteness
\emph{does} coincide with $\sat \Phi$-lex-cocompleteness. 

We now characterise, as promised, those finitely complete $\C$ for which $W
\colon \C \to \app \Phi \C$ admits not just a left adjoint, but a left exact
one. The proof is exactly analogous to that of
Proposition~\ref{prop:charlexcocomp}.
\begin{Prop}\label{prop:charphilexcocomp}Let $\Phi$ be a class of lex-weights, and
let $\C$ be $\sat \Phi$-lex-cocomplete. Then the following are equivalent:
\begin{enumerate}
\item For each lex $D \colon \K \to \C$ with $\K$ small, the functor
    $(\thg) \star D \colon \app \Phi \K \to \C$ is also lex;
\item For each lex $D \colon \K \to \C$, the functor $(\thg) \star D \colon
    \app \Phi \K \to \C$ is also lex;
\item The functor $(\thg) \star 1_\C \colon \app \Phi \C \to \C$ is lex;
\item For each lex $D \colon \K \to \C$ with $\K$ small, the functor
    $\Lan_W D \colon \app \Phi \K \to \C$ is also lex;
\item For each lex $D \colon \K \to \C$, the functor $\Lan_W D \colon \app
    \Phi \K \to \C$ is also lex;
\item The functor $\Lan_W (1_\C) \colon \app \Phi \C \to \C$ is lex;
\item The functor $W \colon \C \to \app \Phi \C$ admits a left exact left
    adjoint;
\item $\C$ admits a structure of $\submonad \Phi$-pseudoalgebra;
\item $\C$ is lex-reflective in some $\app \Phi \D$.
\end{enumerate}
\end{Prop}
We shall call a category $\C$ satisfying any of the equivalent hypotheses of
this proposition \emph{$\Phi$-exact}. Although the definition is given in 
terms of $\sat \Phi$-lex-cocompleteness, rather than $\Phi$-lex-cocompleteness,
it turns out that in the $\Phi$-exact context, the distinction between the two disappears: see Remark~\ref{rk:artificial} below.

We now consider the appropriate notion of morphism between $\Phi$-exact categories. We say that
a functor $F \colon \C \to \D$ between $\Phi$-lex-cocomplete categories is
\emph{$\Phi$-lex-cocontinuous} if it is left exact, and for every $\phi \in
\Phi[\K]$ and lex $D \colon \K \to \C$, the canonical comparison map $\phi
\star FD \to F(\phi \star D)$ is invertible.
\begin{Prop}\label{prop:charphilexmaps}
Let $\C$ and $\D$ be $\Phi$-exact categories, and $F \colon \C \to \D$ a left
exact functor between them. Then the following are equivalent:
\begin{enumerate}
\item $F$ admits a structure of algebra pseudomorphism with respect to some
    (equivalently, any) choice of $\submonad \Phi$-pseudoalgebra structure
    on $\C$ and $\D$;
\item The natural transformation
\begin{equation*}
\cd{
    \app \Phi \C \ar[d]_{(\thg) \star 1_\C} \ar[r]^{\app \Phi F} \dtwocell{dr}{\alpha} & \app \Phi \D \ar[d]^{(\thg) \star 1_\D} \\
    \C \ar[r]_{F} & \D}
\end{equation*}
obtained as the mate under the adjunctions $(\thg) \star 1_\C \dashv W_\C$
and $(\thg) \star 1_\D \dashv W_\D$ of the equality $\app \Phi F. W_\C =
W_\D.F$, is invertible;
\item $F$ is $\sat \Phi$-lex-cocontinuous;
\item $F$ is $\Phi$-lex-cocontinuous.
\end{enumerate}
\end{Prop}
\begin{proof}
(1) $\Leftrightarrow$ (2) since the pseudomonad $\submonad \Phi$ is
lax-idempotent. To see that (2) $\Rightarrow$ (3), suppose that the displayed
$2$-cell $\alpha$ is invertible; then for any $\phi \in \sat \Phi[\K] = \app
\Phi \K$ and left exact $D \colon \K \to \C$, the canonical map $\phi \star FD
\to F(\phi \star D)$ in $\D$ is, to within isomorphism, the component of
$\alpha$ at $\app \Phi D(\phi)$, and hence invertible: which gives (3). It is
clear that (3) $\Rightarrow$ (4), and so it remains to show that (4)
$\Rightarrow$ (2). It suffices by Proposition~\ref{prop:phicchar} to show that
if $F$ preserves $\Phi$-lex-colimits, then the collection of $\phi \in \app
\Phi \C$ for which $\alpha_\phi$ is invertible contains the representables and
is closed under finite limits and $\Phi$-lex-colimits. The first clause is
immediate; the others follow on observing that the composites around both sides
of the square in~(2) preserve finite limits and $\Phi$-lex-colimits: the only
non-obvious fact being that $\app \Phi F$ preserves $\Phi$-lex-colimits, which
follows on observing that $\P F$ does so, being cocontinuous, and that $\app
\Phi \C$ and $\app \Phi \D$ are closed in $\P \C$ and $\P \D$ under
$\Phi$-lex-colimits.
\end{proof}

We define a functor $F \colon \C \to \D$ between $\Phi$-exact categories to be
\emph{$\Phi$-exact} when it satisfies any of the equivalent conditions of this
proposition. We may now verify the claim made in the Introduction that $\app
\Phi \C$ constitutes a free $\Phi$-exact completion of $\C$. In what follows,
we write $\Phi\text-\cat{EX}$ for the $2$-category of $\Phi$-exact categories,
$\Phi$-exact functors and arbitrary natural transformations.

\begin{Prop}
$\Phi\text-\cat{EX}$ is biequivalent to the $2$-category
$\cat{Ps}\text-\submonad \Phi\text-\cat{Alg}$ of $\submonad
\Phi$-pseudoalgebras, pseudoalgebra morphisms and algebra $2$-cells.
\end{Prop}
\begin{proof}
By Propositions~\ref{prop:charphilexcocomp} and~\ref{prop:charphilexmaps}, the
forgetful $2$-functor $U \colon \cat{Ps}\text-\submonad \Phi\text-\cat{Alg} \to
\cat{LEX}$ factors through $\Phi\text-\cat{EX} \to \cat{LEX}$, as
$V$, say. Since $\submonad \Phi$ is lax-idempotent, $U$ is faithful and locally
fully faithful, whence also $V$. By Proposition~\ref{prop:charphilexmaps}, $V$
is also full, hence $2$-fully faithful; since it is moreover clearly surjective
on objects, it is a biequivalence.
\end{proof}

\begin{Cor}\label{cor:biadjoint}
The forgetful $2$-functor $\Phi\text-\cat{EX} \to \cat{LEX}$ has a left
biadjoint; the unit of this biadjunction at the left exact $\C$ may be taken to
be $W \colon \C \to \app \Phi \C$.
\end{Cor}
Combining this result with Proposition~\ref{prop:phicchar}, we conclude, as was
claimed in the Introduction, that the free $\Phi$-exact completion of the
finitely complete $\C$ may be obtained as the closure of the representables in
$\P \C$ under finite limits and $\Phi$-lex-colimits. In fact, we can be more
precise about the nature of the left biadjoint we have just described.
\begin{Prop}\label{prop:equivalencepseudoinverse}
For every finitely complete $\C$ and $\Phi$-exact category $\E$, the functor
\begin{equation*}
    W^\ast \colon \Phi\text-\cat{EX}(\app \Phi \C, \E) \to \cat{LEX}(\C, \E)
\end{equation*}
induced by precomposition with $W$---which Corollary~\ref{cor:biadjoint}
assures us is an equivalence of categories---has equivalence pseudoinverse
given by left Kan extension along $W$.
\end{Prop}
\begin{proof}
Suppose given $F \colon \C \to \E$ lex. By
Propositions~\ref{prop:charphicocomp} and~\ref{prop:charphilexcocomp}, $\Lan_W
F$ exists and is left exact. Moreover, we have $\Lan_W F \cong J \thg \star F$;
and so $\Lan_W F$ preserves $\Phi$-lex-colimits because $J$ does and taking
colimits is cocontinuous in the weight. Thus $\Lan_W F$ is $\Phi$-exact, and
since $(\Lan_W F)W \cong F$, as $W$ is fully faithful, $\Lan_W$ is an
equivalence pseudoinverse for $W^\ast$ as claimed.
\end{proof}

\section{The embedding theorem}\label{sec:embedding}
In the next section, we shall begin to describe, in elementary terms, what the
notion of $\Phi$-exactness amounts to for some particular choices of $\Phi$. In
doing so, we will make repeated use of one further result, which characterises
the $\Phi$-exact categories in terms of the embeddings they admit.
\begin{Thm}\label{thm:embedding}
Let $\Phi$ be a class of lex-weights, and $\C$ a small $\Phi$-lex-cocomplete
category. Then the following conditions are equivalent:
\begin{enumerate}
\item $\C$ admits a full $\Phi$-lex-cocontinuous embedding into a
    $\V$-topos;
\item $\C$ admits a full $\Phi$-lex-cocontinuous embedding into a
    small-exact category;
\item $\C$ admits a full $\Phi$-lex-cocontinuous embedding into a
    $\Phi$-exact category;
\item $\C$ is $\Phi$-exact.
\end{enumerate}
Moreover, even when $\C$ is not small, we still have (1) $\Rightarrow$ (2)
$\Rightarrow$ (3) $\Rightarrow$ (4).
\end{Thm}
In the statement of this theorem, recall that we defined a \emph{$\V$-topos} to
be any category lex-reflective in a presheaf category.
\begin{proof}
We begin by showing that (1) $\Rightarrow$ (2) $\Rightarrow$ (3) $\Rightarrow$
(4), regardless of $\C$'s size. The first two implications are straightforward,
since every $\V$-topos is small-exact by Proposition~\ref{prop:vtopsmallexact},
whilst every small-exact category is clearly $\Phi$-exact. For (3)
$\Rightarrow$ (4), let there be given a $\Phi$-lex-cocontinuous embedding $J
\colon \C \to \E$ into a $\Phi$-exact category. By replacing $\C$ by its
replete image in $\E$, we may assume that $J$ exhibits $\C$ as a full, replete,
finite-limit- and $\Phi$-lex-colimit-closed subcategory of $\E$. Now by
Proposition~\ref{prop:equivalencepseudoinverse}, since $\E$ is $\Phi$-exact,
the left Kan extension $\Lan_W J \colon \app \Phi \C \to \E$ exists and is
$\Phi$-exact. We claim that $\Lan_W J$ factors through the subcategory $\C$;
given this, the factorisation $\app \Phi \C \to \C$ will be $\Lan_W(1_\C)$, and
left exact, since $\Lan_W J$ is, whence $\C$ will be $\Phi$-exact by
Proposition~\ref{prop:charphilexcocomp}(6). To prove the claim, observe that
the collection of $\phi \in \app \Phi \C$ at which $\Lan_W J$ lands in $\C$
contains the representables, since $W^\ast.\Lan_W \cong 1$, and is closed under
finite limits and $\Phi$-lex-colimits, since $\Lan_W J$ preserves them, and
$\C$ is closed in $\E$ under them. This proves (3) $\Rightarrow$ (4); it
remains to show that, when $\C$ is small, we have (4) $\Rightarrow$ (1). We
shall in fact defer this task until Section~\ref{sec:relative} below. There, we
will see that any small $\Phi$-exact $\C$ admits a \emph{small-exact
completion} $V \colon \C \to \appr \P \Phi \C$, and this $V$ will provide the
required full $\Phi$-exact embedding of $\C$ into a $\V$-topos; see
Corollary~\ref{prop:requiredembedding}.
\end{proof}

An obvious limitation of this theorem is that its full strength is only
available for a small $\C$. The following result allows us to work around this;
though it does so at the cost of introducing a further size constraint, this
time on $\Phi$. We call a class of lex-weights $\Phi$ \emph{small} if $\app
\Phi \C$ is small whenever $\C$ is (this condition was called \emph{locally
small} in~\cite{Kelly2005Notes}). This will certainly be the case if the
cardinality of $\Phi$ is small, as is easily seen upon giving a transfinite
construction of $\app \Phi \C$ from $\C$ in the manner
of~\cite[\S3.5]{Kelly1982Basic}.
\begin{Prop}\label{prop:smallsubphiexact}
Let $\Phi$ be a small class of lex-weights, and $\C$ a $\Phi$-lex-cocomplete
category. Now $\C$ is $\Phi$-exact if and only if every small, full,
finite-limit- and $\Phi$-lex-colimit-closed subcategory of $\C$ is
$\Phi$-exact.
\end{Prop}
\begin{proof}
For brevity's sake, let us temporarily agree to call any $\D \subseteq \C$ as
in the statement of this proposition a \emph{$\Phi$-subcategory} of $\C$. By
Theorem~\ref{thm:embedding}, if $\C$ is $\Phi$-exact then so are all of its
$\Phi$-subcategories. Conversely, suppose that every $\Phi$-subcategory of $\C$
is $\Phi$-exact: to show that $\C$ is too, it is enough by
Proposition~\ref{prop:charphilexcocomp}(4) to show that, for every small
finitely complete $\K$ and lex $F \colon \K \to \C$, the Kan extension $\Lan_W
F \colon \app \Phi \K \to \C$ exists and is lex. Given such an $F$, we let $\D$
denote the closure of its image in $\C$ under finite limits and
$\Phi$-lex-colimits, and write
\begin{equation*}
    \K \xrightarrow G \D \xrightarrow H \C
\end{equation*}
for the induced factorisation. Since $\Phi$ and $\K$ are small, so is $\D$; it
is therefore a $\Phi$-subcategory of $\C$ and so $\Phi$-exact by assumption.
Thus $\Lan_W G$ exists and is left exact; whence $H.\Lan_W G$ is also left
exact, and so we will be done if we can show that it is in fact $\Lan_W F$.
Equivalently, we may show that $H$ preserves $\Lan_W G$; equivalently, that for
each $\phi \in \app \Phi \K$, the colimit $\phi \star G$ in $\D$ is preserved
by $H$; or equivalently, that for each $\phi  \in \app \Phi \K$ and $X \in \C$,
the canonical morphism
\begin{equation}\label{eq:kanpropeqa}
    \C(H(\phi \star G), X) \to [({\app \Phi \K})^\op, \V](\phi, \C(HG\thg, X))
\end{equation}
is invertible in $\V$. To do this last, we let $\D'$ be the closure of
%
%
%
%
%
$\D \cup \{X\}$ in $\C$ under finite limits and $\Phi$-lex-colimits. As before,
$\D'$ is a $\Phi$-subcategory of $\C$, whence $\Phi$-exact; moreover, the
inclusion $I \colon \D \to \D'$ preserves finite limits and $\Phi$-lex-colimits
and so by Proposition~\ref{prop:charphilexmaps} is a $\Phi$-exact functor. In
particular, $I$ preserves the colimit $\phi \star G$, which is to say that the
canonical morphism
\begin{equation*}
    \D'(I(\phi \star G), X) \to [({\app \Phi \K})^\op, \V](\phi, \D'(IG\thg, X))
\end{equation*}
is invertible. But this is equally well the morphism~\eqref{eq:kanpropeqa},
since $\D'$ is a full subcategory
of $\C$; thus $H$ preserves $\phi \star G$ for all $\phi \in \app \Phi \K$, so that $H.\Lan_W G$ is $\Lan_W F$ as required. 
%
\end{proof}
The typical manner in which we make use of this result is as follows. Given a
small class of lex-weights $\Phi$, we determine, by some means, a property $Q$
of $\Phi$-lex-cocomplete categories which we believe to be equivalent to
$\Phi$-exactness. We then prove that a \emph{small} $\Phi$-lex-cocomplete $\C$
is $\Phi$-exact if and only if it is $Q$ using Theorem~\ref{thm:embedding}. In
light of Proposition~\ref{prop:smallsubphiexact}, we may then remove the
smallness qualification on $\C$ so long as we can show that a
$\Phi$-lex-cocomplete $\C$ is $Q$ if and only if each of its small, full,
finite-limit- and $\Phi$-lex-colimit-closed subcategory is $Q$: and this will
usually be straightforward, by virtue of the conditions which constitute $Q$
involving quantification only over small sets of data in the candidate category
$\C$.

The size constraint placed on $\Phi$ by this result is relatively harmless,
since most classes of lex-weights that we encounter in practice are in fact
small. However, this is by no means universally so---for instance, the classes
of lex-weights for small coproducts or for small unions of subobjects are not
small---and in order to deal with such cases as these, we now describe a result
allowing the size restriction on $\Phi$ to be circumvented in its turn. It will
be convenient to defer the proof of this result until we have set up the
machinery of small-exact completions; it is given as
Proposition~\ref{prop:setsingletonexact2} below.
\begin{Prop}\label{prop:setsingletonexact}
Let $\Phi$ be a class of lex-weights. A category $\C$ is $\Phi$-exact if
and only if it is $\{\phi\}$-exact for each $\phi \in \Phi$.
\end{Prop}
Assembling the above results, we obtain an embedding theorem for $\Phi$-exact
categories that is subject to no smallness constraints whatsoever.
\begin{Cor}
Let $\Phi$ be a class of lex-weights. A $\Phi$-lex-cocomplete
category $\C$ is $\Phi$-exact if and only if, for each $\phi \in \Phi$,
every small, full, finite-limit- and $\{\phi\}$-lex-colimit-closed subcategory
$\D \subseteq \C$ admits a full $\{\phi\}$-lex-cocontinuous embedding in a
$\V$-topos.
\end{Cor}
\begin{proof}
Combine Theorem~\ref{thm:embedding} with
Propositions~\ref{prop:smallsubphiexact} and~\ref{prop:setsingletonexact}.
\end{proof}
\begin{Rk}\label{rk:artificial}This result characterises the $\Phi$-exact categories as
being $\Phi$-lex-cocomplete categories verifying certain additional conditions;
which is by contrast to Proposition~\ref{prop:charphilexcocomp}, which
characterised them as $\sat \Phi$-lex-cocomplete categories verifying certain
additional conditions. This may seem at odds with the remarks made following
Proposition~\ref{prop:charphicocomp}, where we observed that $\sat
\Phi$-lex-colimits need not always be constructible from $\Phi$-lex-colimits.
However, it turns out that in a $\Phi$-exact $\C$, all $\sat \Phi$-lex-colimits
may in fact be constructed from $\Phi$-lex-colimits \emph{together with finite
limits}. This is possible because the additional conditions verified in a
$\Phi$-exact category force certain cocones under $\sat \Phi$-lex-diagrams,
always constructible from $\Phi$-lex-colimits and finite limits, to be
colimiting ones.
\end{Rk}

\section{Examples of $\Phi$-exactness}\label{sec:3}
We now describe in detail some particular notions of $\Phi$-exactness. As we
have already said, we restrict attention in this article to the unenriched
case---that is, the case $\V = \cat{Set}$ of our general notions---reserving
for future study the consideration of exactness notions over other bases. Thus,
throughout this section and the next, we assume without further comment that
$\V = \cat{Set}$; so ``category'' now means ``locally small category'' and so
on. The examples of this section will show---as anticipated in the
Introduction---that in this setting, and for suitable choices of $\Phi$, a
category is $\Phi$-exact just when it is regular, or Barr-exact, or lextensive,
or coherent, or adhesive. We also provide three further examples fitting into our
framework. The first is the notion of category with stable and effective finite
unions of subobjects (effectivity meaning that unions are calculated as a
pushout over the pairwise intersections); the second and third are the appropriate notions
of exactness for categories with filtered colimits, and for categories with reflexive coequalisers.
\subsection{Regular categories}
For our first example, let the class $\Phi_\mathrm{reg}$ be given by the single
functor $\phi \colon \K^\op \to \cat{Set}$, where $\K$ is the free category
with finite limits generated by an arrow $f \colon X \to Y$ and where $\phi$ is
the coequaliser in $[\K^\op, \cat{Set}]$ of the kernel-pair of $\K(\thg, f)
\colon \K(\thg, X) \to \K(\thg, Y)$; note that $\phi$ is equally well the image
of $\K(\thg, f)$. If $(s,t) \colon R \rightrightarrows X$ is the
kernel-pair of $f$ in $\K$, then, since the Yoneda embedding preserves limits,
$\phi$ is equally well a coequaliser of $\K(\thg, s)$ and $\K(\thg, t)$ in
$[\K^\op, \cat{Set}]$. Now suppose given a finitely complete $\C$ and a lex functor $D
\colon \K \to \C$. Since colimits by a representable weight are given by evaluation at the representing object, and since the weighted colimit functor is cocontinuous in its first argument, insofar as it is defined, the colimit $\phi \star D$, if it exists, must be a
coequaliser of the pair $(Ds, Dt) \colon DR \rightrightarrows DX$. But as
$D$ preserves finite limits, this pair is a kernel-pair of $Df$, whence $\phi
\star D$ must be the coequaliser of the kernel-pair of $Df$. Thus if $\C$ admits
coequalisers of kernel-pairs, it is $\Phi_\mathrm{reg}$-lex-cocomplete;
conversely, if $\C$ is $\Phi_\mathrm{reg}$-lex-cocomplete, then it admits
coequalisers of kernel-pairs, since for any $h \colon U \to V$ in $\C$ there is
some lex $D \colon \K \to \C$ with $Df = h$.

Now by Theorem~\ref{thm:embedding}, a small, finitely complete and
$\Phi_\mathrm{reg}$-lex-cocomplete $\C$ is $\Phi_\mathrm{reg}$-exact just when
it admits a full embedding into a Grothendieck topos preserving finite limits
and coequalisers of kernel-pairs; equivalently, finite limits and regular
epimorphisms.
Such an embedding, being fully faithful, will reflect as well as preserve
regular epimorphisms, and since regular epimorphisms in a Grothendieck topos
are stable under pullback, it follows that the same is true in any small
$\Phi_\mathrm{reg}$-exact category: which is to say that any such category is
regular. Conversely, if $\C$ is small and regular, then we may consider the
topos $\cat{Sh}(\C)$ of sheaves on $\C$ for the \emph{regular topology}, in
which a sieve is covering just when it contains some regular epimorphism.
By~\cite[Proposition 4.3]{Barr1971Exact}, the canonical functor $\C \to
\cat{Sh}(\C)$ is fully faithful, and preserves both finite limits and regular
epimorphisms; whence $\C$ is $\Phi_\mathrm{reg}$-exact. Thus the small,
finitely complete $\C$ is $\Phi_\mathrm{reg}$-exact if and only if regular; and
since clearly a category with finite limits and coequalisers of kernel-pairs is
regular if and only if every small, full subcategory closed under finite limits
and coequalisers of kernel-pairs is regular, we conclude from
Proposition~\ref{prop:smallsubphiexact} that a finitely complete $\C$, of any
size, is $\Phi_\mathrm{reg}$-exact if and only if it is regular.

\subsection{Barr-exact categories}\label{subsec:barrexact}
Consider now the class of lex-weights $\Phi_\mathrm{ex}$ consisting of the
single functor $\phi \colon \K^\op \to \cat{Set}$, where $\K$ is the free
category with finite limits generated by an equivalence relation $(s,t) \colon
R \rightarrowtail X \times X$ and where $\phi$ is the coequaliser in $[\K^\op,
\cat{Set}]$ of $\K(\thg, s)$ and $\K(\thg, t)$. Arguing as before, we see that
the finitely complete $\C$ is $\Phi_\mathrm{ex}$-lex-cocomplete if and only if
it admits coequalisers of equivalence relations. Now by
Theorem~\ref{thm:embedding} such a $\C$, if small, is $\Phi_\mathrm{ex}$-exact
just when it admits a fully faithful functor $J \colon \C \to \E$ into a
Grothendieck topos which preserves finite limits and coequalisers of
equivalence relations. Since any kernel-pair is an equivalence relation, such a
$J$ in particular preserves and reflects regular epimorphisms, and so any small
$\Phi_\mathrm{ex}$-exact category is regular. If moreover $(s,t) \colon R
\rightarrowtail X \times X$ is an equivalence relation in $\C$, then by virtue
of $J$'s preserving coequalisers of equivalence relations, and reflecting kernel-pairs, we conclude that
$(s,t)$ is the kernel-pair of its coequaliser, since $(Js,Jt)$ is so in the
topos $\E$. Thus any small $\Phi_\mathrm{ex}$-exact category is Barr-exact.
Conversely, if the small, finitely complete $\C$ is Barr-exact, then the
embedding $\C \to \cat{Sh}(\C)$---where $\C$ is again equipped with the regular
topology---preserves not only regular epimorphisms but also coequalisers of
equivalence relations, since every equivalence relation in $\C$ and in
$\cat{Sh}(\C)$ is the kernel-pair of its own coequaliser. Thus the small
finitely complete $\C$ is $\Phi_\mathrm{ex}$-exact if and only if Barr-exact;
and so appealing to Proposition~\ref{prop:smallsubphiexact} and arguing as
before, we conclude that the $\Phi_\mathrm{ex}$-exact categories of any size
are precisely the finitely complete Barr-exact categories. It follows from this
that if $\C$ is finitely complete, then $W \colon \C \to \Phi_\mathrm{ex} \C$
is what is usually referred to as the ex/lex completion of $\C$, as described
explicitly in~\cite{Carboni1995Some}. The fact that $\C$ is itself Barr-exact
just when $W$ admits a left exact left adjoint---which is immediate from our
Proposition~\ref{prop:charphilexcocomp}(7)---was first noted in~\cite[Lemma
2.1(iv)]{Carboni1995Some}; our theory provides a general context for this
observation.

\subsection{Lextensive categories}
Consider next the class of lex-weights $\Phi_\mathrm{lext}$ consisting of the
two functors $\phi_0 \colon \K_0^\op \to \cat{Set}$ and $\phi_2 \colon \K_2^\op
\to \cat{Set}$. Here, $\K_0$ is the terminal category, and $\phi_0$ the initial
object of $[\K_0^\op, \cat{Set}]$, whilst $\K_2$ is the free category with
finite limits on a pair of objects $X,Y$, and $\phi_2$ the coproduct
$\K_2(\thg, X) + \K_2(\thg, Y)$. Arguing as before, a finitely complete $\C$ is
$\Phi_\mathrm{lext}$-lex-cocomplete if and only if it admits finite coproducts.

In order to characterise the $\Phi_\mathrm{lext}$-exact categories, we shall
describe directly the free $\Phi_\mathrm{lext}$-exact category on a finitely
complete $\C$. Let $\cat{Fam}_f(\C)$ be the finite coproduct completion of
$\C$; its objects are finite collections $(X_i \mid i \in I)$ of objects of
$\C$ whilst its morphisms $(X_i \mid i \in I) \to (Y_j \mid j \in J)$ are pairs
of a function $f \colon I \to J$ and a family of morphisms $(g_i \colon X_i \to
Y_{f(i)} \mid i \in I)$. We have fully faithful functors $W \colon \C \to
\cat{Fam}_f(\C)$ and $J \colon \cat{Fam}_f(\C) \to \P \C$, where $W(X) = (X)$
and $J(X_i \mid i \in I) = \sum_{i} \C(\thg, X_i)$, and clearly have $Y \cong
JW$. Of course, $\cat{Fam}_f(\C)$ has finite coproducts; it is also finitely
complete, as remarked in~\cite[Lemma 4.1(ii)]{Carboni1995Some}, and both the
finite coproducts and the finite limits are easily seen to be preserved by $J$.
Moreover, every object of $\cat{Fam}_f(\C)$ is a finite coproduct of objects in
the image of $W$, and thus the replete image of $J$ in $\P \C$ is precisely
$\Phi_\mathrm{lext} \C$. It follows that a category $\C$ with finite limits and
finite coproducts is $\Phi_\mathrm{lext}$-exact just when the functor
$\cat{Fam}_f(\C) \to \C$ sending $(X_i \mid i \in I)$ to $\sum_{i \in I} X_i$
preserves finite limits; which by~\cite[Theorem 9]{Rosebrugh1986Cofibrations},
will happen just when finite coproducts in $\C$ are stable and disjoint. Thus we conclude that a
category $\C$ is $\Phi_\mathrm{lext}$-exact just when it is lextensive.

\subsection{Effective unions}
Let $\Phi_\vee$ be given by the two functors $\phi_0 \colon \K_0^\op \to
\cat{Set}$ and $\phi_2 \colon \K_2^\op \to \cat{Set}$ defined as follows.
$\K_0$ is the terminal category, and $\phi_0$ the initial object of $[\K_0^\op,
\cat{Set}]$; $\K_2$ is the free category with finite limits generated by a pair
of monomorphisms $A \rightarrowtail C \leftarrowtail B$, and $\phi_2$ is the
union in $[\K_2^\op, \cat{Set}]$ of the two subobjects $\K_2(\thg, A)$ and
$\K_2(\thg, B)$ of $\K_2(\thg, C)$. Writing $A \cap B$ for the intersection of
the subobjects $A$ and $B$ of $C$ in $\K_2$, we observe that $\phi_2$ is
equally well the pushout of the inclusions of $\K_2(\thg, A \cap B)$ into
$\K_2(\thg, A)$ and $\K_2(\thg, B)$. It follows from this that a finitely
complete category $\C$ is $\Phi_\vee$-lex-cocomplete just when it admits an
initial object and pushouts of pullbacks of pairs of monomorphisms. The
following result characterises the $\Phi_\vee$-exact categories.
\begin{Prop}\label{prop:effunions}
The following are equivalent properties of the finitely complete $\C$:
\begin{enumerate}
\item $\C$ is $\Phi_\vee$-exact;
\item $\C$ has a strict initial object, and pushouts of pullbacks of pairs
    of monomorphisms which are stable under pullback;
\item $\C$ admits finite unions of subobjects which are effective and
    stable under pullback.
\end{enumerate}
\end{Prop}
For part (3), finite unions are said to be \emph{effective} if for any $C \in
\C$ and any pair of
    subobjects $A, B$ of $C$, the pullback square
\begin{equation*}
    \cd[@!@-1.9em]{
        A \cap B \ar@{ >->}[r] \ar@{ >->}[d] & B \ar@{ >->}[d] \\
        A \ar@{ >->}[r] & A \cup B
    }
\end{equation*}
is also a pushout.\begin{proof} By appealing to
Proposition~\ref{prop:smallsubphiexact}, and arguing as before, it suffices to
prove the equivalence when $\C$ is small. So suppose first that the small $\C$
satisfies (1). By Theorem~\ref{thm:embedding} we know that $\C$ is
$\Phi_\vee$-lex-cocomplete and admits a full embedding $J \colon \C \to \E$
where $\E$ is a Grothendieck topos and $J$ preserves finite limits, the initial
object, and pushouts of pullbacks of pairs of monomorphisms. Thus $\C$'s
initial object is strict, since if $f \colon X \to 0$ in $\C$ then $Jf$ is
invertible in $\E$---as initial objects in a topos are strict---whence $f$ is
also invertible, as $J$ is conservative. A similar argument shows that pushouts
of pullbacks of monomorphisms are stable under pullback in $\C$, since they are
so in $\E$. Thus (1) $\Rightarrow$ (2).

Suppose next that $\C$ satisfies (2). Strictness of the initial object $0$
implies that the unique map $0 \to C$ is always monomorphic; whence $0
\rightarrowtail C$ is a least subobject of $C$, which by strictness is stable
under pullback. If now $A$ and $B$ are subobjects of $C$, we claim that the map
$k \colon A +_{A \cap B} B \to C$ is a monomorphism; if this is so, then $k$
represents the subobject $A \cup B$ of $C$, and such binary unions are stable
by assumption, and effective by construction. The claim is proved in
Theorem~5.1 of~\cite{Lack2005Adhesive}; we give here an alternative proof.
Observe first that whenever we pull back the diagram
\begin{equation*}
    \cd[@-0.9em@C-2em]{
        & A \cap B \ar[dl] \ar[dr] \\
        A \ar[dr] \ar@/_14pt/[ddr] & & B \ar[dl] \ar@/^14pt/[ddl] \\
        & A +_{A \cap B} B \ar[d]^k \\
        & C
    }
\end{equation*}
along a map $f \colon K \to C$, the inner square remains a pushout by
stability; so that if the outer square becomes a pushout, the induced map
$f^\ast(k)$ must be an isomorphism. In particular, this is the case when $f$ is
any of the inclusions $A, B, A \cap B \rightarrowtail C$, so that on
considering the subobjects $k^\ast(A)$, $k^\ast(B)$ and $k^\ast(A \cap B)$ of
$A +_{A \cap B} B$, we have the comparison maps $k^\ast(A) \to A$, $k^\ast(B)
\to B$ and $k^\ast(A \cap B) \to A \cap B$ invertible. But this in turn implies
that on pulling back the displayed diagram along $k$, the outer square becomes
a pushout; whence $k^\ast(k)$ is invertible, so that $k$ has trivial
kernel-pair and is thereby monomorphic. This completes the proof of the claim,
and so (2) $\Rightarrow$ (3).

Suppose now that $\C$ satisfies (3). Clearly $\C$ has pushouts of pullbacks of
monomorphisms, and a standard argument shows that it also has an
initial object---see~\cite[Lemma A1.4.1]{Johnstone2002Sketches}, for example.
Thus $\C$ is $\Phi_\vee$-lex-cocomplete; to show that it is in fact
$\Phi_\vee$-exact, we define a topology on $\C$ by declaring that $C \in \C$ is
covered by any sieve containing a finite family of subobjects $(A_i
\rightarrowtail C \mid i \in I)$ whose union is all of $C$. Stability of finite
unions ensures that this gives a topology on $\C$, whilst stability and
effectivity together ensure that this topology is subcanonical. So we have a
restricted Yoneda embedding $\C \to \cat{Sh}(\C)$ into the category of sheaves
for this topology, which is fully faithful and left exact; to complete the
proof, it is enough to show that it is also $\Phi_\vee$-lex-cocontinuous, or
equivalently, that every sheaf $F \colon \C^\op \to \cat{Set}$ is
$\Phi_\vee$-lex-continuous. Now if $F$ is such a sheaf, then certainly $F(0)
\cong 1$, since the empty family covers $0$; it remains to show that if $A
\rightarrowtail C \leftarrowtail B$ in $\C$, then the square
\begin{equation*}
\cd{    F(A +_{A \cap B} B) \ar[r] \ar[d] & FA \ar[d] \\ FB \ar[r] & F(A \cap B)}
\end{equation*}
is a pullback in $\cat{Set}$. But $A +_{A \cap B} B = A \cup B$ by assumption,
so this follows from the sheaf condition applied to the covering family $A
\rightarrowtail A \cup B \leftarrowtail B$.
\end{proof}
The class $\Phi_\vee$ admits an obvious generalisation to a class
$\Phi_{\bigvee}$ for which the $\Phi_{\bigvee}$-exact categories may be
characterised as those with effective, stable unions of small families of
subobjects. We do not take the trouble to formulate this precisely; though let
us observe that, since the class of lex-weights $\Phi_{\bigvee}$ will no longer
be small, we must make use of Proposition~\ref{prop:setsingletonexact} in
proving the characterisation.

\subsection{Coherent and geometric categories}
Consider now the class of lex-weights $\Phi_\mathrm{coh} = \Phi_\mathrm{reg}
\cup \Phi_\mathrm{\vee}$. We deduce from
Proposition~\ref{prop:setsingletonexact} that a finitely complete category is
$\Phi_\mathrm{coh}$-exact just when it is both $\Phi_\mathrm{reg}$-exact and
$\Phi_\mathrm{\vee}$-exact; that is, just when it is regular and admits stable
effective finite unions of subobjects. In fact, if a regular category admits
stable finite unions (and recall that such a category is called
\emph{coherent}), they are necessarily effective: see, for
example~\cite[Proposition A1.4.3]{Johnstone2002Sketches}. Hence a finitely
complete category is $\Phi_\mathrm{coh}$-exact just when it is a coherent
category.

On the other hand, we may consider the class of lex-weights
$\Phi'_\mathrm{coh}$ comprising the two functors $\phi_0 \colon \K_0^\op \to
\cat{Set}$ and $\phi_2 \colon \K_2^\op \to \cat{Set}$ defined as follows.
$\K_0$ is the terminal category, and $\phi_0$ the initial object of $[\K_0^\op,
\cat{Set}]$; $\K_2$ is the free category with finite limits generated by a pair
of arrows $A \to C \leftarrow B$, and $\phi_2$ is the image in $[\K_2^\op,
\cat{Set}]$ of the copairing $\K_2(\thg, A) + \K_2(\thg, B) \to \K_2(\thg, C)$.
By an argument similar to the preceding ones, a category $\C$ is
$\Phi'_\mathrm{coh}$-lex-cocomplete just when it has an initial object, and for
every cospan $f \colon A \to C \leftarrow B \colon g$ in $\C$, the diagram
\begin{equation*}
    \cd[@C-1em@-0.5em]{
      A \times_C A \ar@<-3pt>[dr] \ar@<3pt>[dr] & & A \times_C B \ar[dl] \ar[dr] & & B \times_C B \ar@<-3pt>[dl] \ar@<3pt>[dl] \\
      & A & & B
    }
\end{equation*}
---which for the purposes of this example we will call the \emph{double kernel}
of $(f,g)$---admits a colimit. In particular, such a $\C$ admits pushouts of
pullbacks of pairs of monomorphisms (take $f$ and $g$ monomorphic), and
coequalisers of kernel-pairs (take $f = g$), so that by
Proposition~\ref{prop:phicchar}, $\Phi_\mathrm{coh} \C \subseteq
\Phi'_\mathrm{coh} \C$ for every finitely complete $\C$, whence any
$\Phi'_\mathrm{coh}$-exact category is $\Phi_\mathrm{coh}$-exact and so
coherent. We claim conversely that every coherent category $\C$ is
$\Phi'_\mathrm{coh}$-exact. By Proposition~\ref{prop:smallsubphiexact}, it
suffices to consider the case of a small $\C$. We may equip such a $\C$ with
the \emph{coherent topology}, in which a sieve is covering just when it
contains a finite family of morphisms $(A_i \to X \mid i \in I)$ whose images
have as union the whole of $X$. This topology is subcanonical, and so we have a
fully faithful and left exact embedding $J \colon \C \to \cat{Sh}(\C)$. Since
the empty family covers $0$, this embedding preserves the initial object; we
must show it also preserves colimits of double kernels. So let $f \colon A \to
C \leftarrow B \colon g$ in $\C$, and define $Z
\defeq \mathrm{im} f \cup \mathrm{im} g \rightarrowtail C$. Now the induced maps
$f' \colon A \to Z \leftarrow B \colon g'$ exhibit $Z$ as the colimit of the
double kernel of $(f,g)$; but since $Z$ is a subobject of $C$, this double
kernel is equally that of $(f',g')$, whose colimit $J$ preserves since the pair
$(f', g')$ covers $Z$. Thus $J$ preserves $\Phi'_\mathrm{coh}$-lex-colimits;
and so every coherent category is $\Phi'_\mathrm{coh}$-exact as claimed.

In an analogous way, we may also formulate classes $\Phi_\mathrm{geom} =
\Phi_\mathrm{reg} \cup \Phi_{\bigvee}$ and $\Phi'_\mathrm{geom}$ such that a
category is $\Phi_\mathrm{geom}$-exact if and only if it is
$\Phi'_\mathrm{geom}$-exact, if and only if it is a \emph{geometric}
category---that is, a regular category with pullback-stable small unions of
subobjects.

\subsection{Adhesive categories}\label{subsec:adhesive}
Consider the class of lex-weights $\Phi_\mathrm{adh}$ comprising the single
functor $\phi \colon \K^\op \to \cat{Set}$, where $\K$ is the free category
with finite limits generated by a span $m \colon A \leftarrowtail C \to B
\colon f$ with $m$ monomorphic, and where $\phi$ is the pushout in $[\K^\op,
\cat{Set}]$ of $\K(\thg, m)$ and $\K(\thg, f)$. Now a finitely complete
category $\C$ is $\Phi_\mathrm{adh}$-lex-cocomplete just when it admits
pushouts along monomorphisms. Recall from~\cite{Lack2005Adhesive} that we call
such a category \emph{adhesive} when for any commutative cube
\begin{equation*}
    \cd[@-1.3em]{
        \cdot \ar[rr] \ar[dr] \ar[dd] & & \cdot \ar[dd]|\hole \ar[dr]\\
        & \cdot \ar[rr] \ar[dd] & & \cdot \ar[dd] \\
        \cdot \ar[rr]|\hole \ar[dr] & & \cdot \ar[dr]\\
        & \cdot \ar[rr] & & \cdot
    }
\end{equation*}
whose bottom face is a pushout and whose rear faces are pullbacks, the top face
is a pushout if and only if the front faces are pullbacks. This condition
implies, in particular, that pushouts along monomorphisms are stable under
pullback, and that every such pushout square is a pullback. In fact, these
consequences of adhesivity turn out to be equivalent to it: a direct proof is
 given in~\cite{GarnerOn-the-axioms}, but the result may also be
deduced from our general theory.
\begin{Prop}\label{prop:adhchar}
The following are equivalent properties of the finitely complete $\C$:
\begin{enumerate}
\item $\C$ is $\Phi_\mathrm{adh}$-exact;
\item $\C$ is adhesive;
\item $\C$ admits pushouts along monomorphisms which are stable under
    pullback; moreover, every such pushout square is a pullback.
\end{enumerate}
\end{Prop}
\begin{proof}
By Proposition~\ref{prop:smallsubphiexact} we may assume, as before, that $\C$
is small. Now if $\C$ is $\Phi_\mathrm{adh}$-exact then by
Theorem~\ref{thm:embedding} it admits a full embedding into a Grothendieck
topos which preserves finite limits and pushouts along monomorphisms. Since
such an embedding also reflects finite limits, and since every Grothendieck
topos is adhesive, either by~\cite{Lack2006Toposes} or by a simple direct
argument, it follows that $\C$ is adhesive; and so (1) $\Rightarrow$ (2). On
the other hand, Theorem~3.3 of~\cite{Lack2011An-embedding} shows that every small
adhesive category admits a full embedding into a Grothendieck topos which
preserves finite limits and pushouts along monomorphisms; so by
Theorem~\ref{thm:embedding}, we have (2) $\Rightarrow$ (1). Next, if $\C$ is
adhesive, then pushouts along monomorphisms are certainly stable under
pullback, as this is one half of the defining property of adhesivity. Moreover,
every such pushout square is a pullback by~\cite[Lemma 4.3]{Lack2005Adhesive}:
and thus (2) $\Rightarrow$ (3).

To complete the proof, it remains  to show either (3) $\Rightarrow$ (2) or (3)
$\Rightarrow$ (1). As mentioned above, it turns out that there is a direct,
elementary argument for the first of these, which is given
in~\cite{GarnerOn-the-axioms}. But we do not need it here; for a close examination
of the proof of (2)~$\Rightarrow$~(1) provided by~\cite{Lack2011An-embedding}
reveals that it is actually a proof of (3) $\Rightarrow$ (1). It requires no
more than that pushouts along monomorphisms are stable under pullback, that
such pushouts are also pullbacks, and that monomorphisms are stable under
pushout. We have assumed all of these in (3) except the last; but this follows
on observing that, if
\begin{equation*}
    \cd{
    C \ar[r]^f \ar[d]_m & B \ar[d]^n \\ A \ar[r]_g & D
    }
\end{equation*}
is a pushout with $m$ monomorphic, then it is a pullback by assumption, so that
on pulling back the whole square along $n$, its left edge becomes invertible.
Since the resultant square is again a pushout, its right edge must also be
invertible, which is to say that $n$ has trivial kernel-pair and so is
monomorphic.
\end{proof}

\subsection{Filtered colimits}
For our next example, we let $\Phi_\mathrm{filt}$ be the class of lex-weights $\{\,\phi_\K \colon \F(\K)^\op \to \cat{Set} \mid\text{$\K$ a small filtered category}\,\}$; here, $\F(\K)$ is a free completion of $\K$ under finite limits---with unit $E \colon \K \to \F(\K)$, say---and the presheaf $\phi_\K$ is the left Kan extension along $E^\op$ of the terminal object of $[\K^\op, \cat{Set}]$. Now arguments like those of the preceding sections show that a finitely complete category $\C$ is $\Phi_\mathrm{filt}$-lex-cocomplete just when it is filtered-cocomplete. We claim moreover that:
\begin{Prop}
A finitely complete category $\C$ is $\Phi_\mathrm{filt}$-exact just when it admits filtered colimits and these commute with finite limits.
\end{Prop}
\begin{proof}
We argue as in the previous examples; however, a little extra care is needed since the class of lex-weights $\Phi_\mathrm{filt}$ is not small. For every regular cardinal $\kappa$, let $\Phi_\kappa \subset \Phi_\mathrm{filt}$ denote the class $\{\,\phi_\K \colon \F(\K)^\op \to \cat{Set} \mid \text{$\K$ is $\kappa$-small and filtered}\,\}$. Clearly $\Phi_\mathrm{filt} = \bigcup_\kappa \Phi_\kappa$ and so by Proposition~\ref{prop:setsingletonexact}, a finitely complete and filtered-cocomplete $\C$ is $\Phi_\mathrm{filt}$-exact just when it is $\Phi_\kappa$-exact for each $\kappa$. Moreover, finite limits commute with filtered colimits in $\C$ just when finite limits commute with $\kappa$-small filtered colimits for each regular $\kappa$, and so to complete the proof, it will suffice to show that for each $\kappa$, a category $\C$ with finite limits and $\kappa$-small filtered colimits is $\Phi_\kappa$-exact just when these limits and colimits commute with each other. In fact, since the class $\Phi_\kappa$ is small, it will suffice by Proposition~\ref{prop:smallsubphiexact} to do this  only for the case of a small $\C$.

Now by Theorem~\ref{thm:embedding}, the small $\C$ with finite limits and $\kappa$-small filtered colimits is $\Phi_\kappa$-exact just when it admits a full embedding into a Grothendieck topos $\E$ preserving these limits and colimits. If $\C$ admits such an embedding, then finite limits will commute with $\kappa$-small filtered colimits in it, since they do so in $\E$. Conversely, suppose that finite limits commute with $\kappa$-small filtered colimits in $\C$; then $\kappa$-small filtered colimits are in particular stable under pullback, and so we obtain a subcanonical topology on $\C$ by declaring that the injections into every $\kappa$-small filtered colimit should be a covering family. Let $\cat{Sh}(\C)$ be the category of sheaves for this topology and $J \colon \C \to \cat{Sh}(\C)$ the restricted Yoneda embedding; clearly $J$ preserves finite limits, and we will be done if we can show that it also preserves $\kappa$-small filtered colimits.

Now if $(p_k \colon Dk \to X \mid k \in \K)$ is a $\kappa$-small filtered colimit in $\C$, then $J$ will preserve it just when every sheaf $F \colon \C^\op \to \cat{Set}$  sends it to a limit in $\cat{Set}$. So let $F \colon \C^\op \to \cat{Set}$ be a sheaf; since the family $(p_k \mid k \in \K)$ is covering, we may identify $FX$ with the set of matching families for this covering. In other words, if
\begin{equation*}
\cd[@-0.5em]{
  D_{jk} \ar[r]^-{d_{jk}} \ar[d]_{c_{jk}} & Dj \ar[d]^{p_j} \\
  Dk \ar[r]_{p_k} & X
}  
\end{equation*}
is a pullback for each $j, k \in \K$, then we may identify $FX$ with the set
\begin{equation}\label{eq:theset}\tag{$\ast$}
 \{ \vec x \in \Pi_k FDk \,\mid\,  Fd_{jk}(x_j) = Fc_{jk}(x_k) \text{ for all $j, k \in \K$}\}\rlap{ .}
\end{equation}
Under this identification, the canonical comparison map $FX \to \lim FD$ is just the inclusion between these sets, seen as subobjects of $\Pi_k FDk$, and so injective; to complete the proof, we must show that it is also surjective. Thus we must show that each $\vec x \in \lim FD$ lies in~\eqref{eq:theset}, or in other words, that  $Fd_{jk}(x_j) = Fc_{jk}(x_k)$ for each $\vec x \in \lim FD$ and each $j, k \in J$. 
To this end, we consider the category $\K'$ of cospans from $j$ to $k$ in $\K$; since $\K$ is filtered and $\kappa$-small, it follows easily that $\K'$ is too. We define a functor $E \colon \K' \to \C$ by sending each cospan $f \colon j \to \ell \leftarrow k \colon g$ in $\K'$ to the 
%
%
apex of the pullback square\[
\cd[@-0.5em]{
  E(f,g) \ar[r]^-{u_{f,g}} \ar[d]_{v_{f,g}} & Dj \ar[d]^{Df} \\
  Dk \ar[r]_{Dg} & D\ell
}
\]
in $\C$. A simple calculation shows that
$p_k.v_{f,g} = p_j.u_{f,g}$, so that we have induced maps $q_{f,g} \defeq (u_{f,g}, v_{f,g}) \colon E(f,g) \to D_{jk}$, constituting a cocone $q$ under $E$ with vertex $D_{jk}$. We claim that this cocone is colimiting; whereupon, 
by the preceding part of the argument, the comparison $FD_{jk} \to \lim FE$ induced by $q$ will be  monic, and consequently the family $(Fq_{f,g} \mid (f,g) \in \K')$ will be jointly monic. Thus in order to verify that $Fd_{jk}(x_j) = Fc_{jk}(x_k)$, and so complete the proof, it will be enough to observe that for each cospan $f \colon j \to \ell \leftarrow k \colon g$ in $\K'$, we have:
\begin{align*}Fq_{f,g}(Fd_{jk}(x_j)) &= Fu_{f,g}(x_j) = Fu_{f,g}(FDf(x_\ell)) \\ &= Fv_{f,g}(FDg(x_\ell)) = Fv_{f,g}(x_k) = Fq_{f,g}(Fc_{jk}(x_k))\rlap{ .}
\end{align*}

It remains to show that $q$ is colimiting. For this, let $V \colon \K' \to \K$ denote the functor sending a $j,k$-cospan to its central object, and $\iota_1 \colon \Delta j \to V \leftarrow \Delta k \colon \iota_2$ the evident natural transformations. Now we have a commutative cube
\begin{equation*}
\cd[@!@-3em@C+0.5em]{
  E \ar[rr]^u \ar[dd]_v \ar[dr]^{q} & & \Delta(Dj) \ar[dd]_(0.25){D\iota_1}|\hole \ar@{=}[dr] \\ &
  \Delta(D_{jk}) \ar[rr]_(0.7){\Delta d_{jk}} \ar[dd]^(0.75){\Delta c_{jk}} & &
  \Delta(Dj) \ar[dd]^{\Delta p_j} \\
  \Delta(Dk) \ar[rr]^(0.33){D \iota_2}|\hole \ar@{=}[dr] & & DV \ar[dr]^{pV} \\ &
  \Delta(Dk) \ar[rr]_{\Delta p_k} & & \Delta X
}
\end{equation*}
in $[\K', \C]$; its front and rear faces are pullbacks, and remain so on applying the (conical) colimit functor $[\K', \C] \to \C$, since finite limits commute with $\kappa$-small filtered colimits in $\C$. To show that $q$ is colimiting is equally to show that it is inverted by this  functor;  for which, by the previous sentence, it is enough to show that $pV$ is likewise inverted, or in other words that $pV$ is a colimit cocone. But $\K$'s filteredness implies easily that $V \colon \K' \to \K$ is a final functor, so that $pV$, like $p$, is colimiting as desired.
\end{proof}

\subsection{Reflexive coequalisers}\label{subsec:reflexive}
For our final example, consider the class of lex-weights $\Phi_\mathrm{rc}$ comprising the single
functor $\phi \colon \K^\op \to \cat{Set}$, where $\K$ is the free category
with finite limits generated by a reflexive pair $(d, c) \colon X
\rightrightarrows Y$ (with common splitting $r$, say), and where $\phi$ is the
coequaliser in $[\K^\op, \cat{Set}]$ of $\K(\thg, d)$ and $\K(\thg, c)$. Now a
finitely complete category is $\Phi_\mathrm{rc}$-lex-cocomplete just when it
admits coequalisers of reflexive pairs. The following result characterises the
$\Phi_\mathrm{rc}$-exact categories.
\begin{Prop}\label{prop:reflexivechar}
The following are equivalent properties of the finitely complete $\C$:
\begin{enumerate}
\item $\C$ is $\Phi_\mathrm{rc}$-exact;
\item $\C$ is Barr-exact, and for every reflexive relation $R
    \rightarrowtail X \times X$ in $\C$, the chain $R \subseteq RR^oR
    \subseteq RR^oRR^oR \subseteq \cdots$ of subobjects of $X \times X$ has
    a pullback-stable colimit;
\item $\C$ is Barr-exact, and for every reflexive relation $R
    \rightarrowtail X \times X$ in $\C$, the chain $R \subseteq RR^oR
    \subseteq RR^oRR^oR \subseteq \cdots$ of subobjects of $X \times X$ has
    an effective, pullback-stable union.
\end{enumerate}
\end{Prop}
Observe that, in parts (2) and (3), we employ the calculus of internal
relations in $\C$---see \cite{Carboni1987Cartesian}, for instance---which we
are entitled to do, since $\C$ is Barr-exact, and so in particular regular.

\begin{proof}
The argument that (2) $\Leftrightarrow$ (3) is exactly as in
Proposition~\ref{prop:effunions} above, and so it is enough to show that these
are in turn equivalent to (1). We begin by showing that a $\C$ as in (3) is
$\Phi_\mathrm{rc}$-exact. First we show that such a $\C$ admits coequalisers of
reflexive pairs. The argument is a standard one---given in~\cite[Lemma
A1.4.19]{Johnstone2002Sketches}, for example---and so we indicate only its
adaptation to the particularities of our situation. Given a reflexive pair $(s,
t) \colon Y \to X \times X$, we first form its image $(d, c) \colon R
\rightarrowtail X \times X$: now a coequaliser for the latter will also be one
for the former, as the comparison map $Y \twoheadrightarrow R$ is regular epi.
Since $R$ is a reflexive relation, we may by assumption form the union of the
chain $R \subseteq RR^oR \subseteq RR^oRR^oR \subseteq \cdots$; let us write it
as $(d',c') \colon R^\ast \rightarrowtail X \times X$. By stability, $(d', c')$
is an equivalence relation, and so admits a coequaliser, which it is not hard
to show is also a coequaliser for $(d,c)$, and hence for $(s,t)$. Thus $\C$
admits coequalisers of reflexive pairs; let us record for future use that,
since $\C$ is Barr-exact, the $(d',c')$ of the above argument is also the
kernel-pair of the coequaliser of $(s,t)$.

We now show that $\C$ is $\Phi_\mathrm{rc}$-exact. By
Proposition~\ref{prop:smallsubphiexact}, we may assume that $\C$ is small;
whereupon, by Theorem~\ref{thm:embedding}, it is enough to show that $\C$
admits a fully faithful embedding into a Grothendieck topos which preserves
finite limits and coequalisers of reflexive pairs. We define a topology on $\C$
by declaring that every regular epimorphism should cover its codomain, and
that, for every reflexive relation $R \rightarrowtail X \times X$, the family
of union inclusions
\begin{equation}\label{eq:coveringfamily}
    \cd[@!C]{
        R \ar[dr] & RR^oR \ar[d] & RR^oRR^oR \ar[dl] & \dots \\ & R^\ast
    }
\end{equation}
should cover $R^\ast$. By assumption, these covers are effective-epimorphic and
stable under pullback, and so generate a subcanonical topology on $\C$. Thus
there is a full, lex embedding $J \colon \C \to \cat{Sh}(\C)$, and we will be
done if we can prove that $J$ preserves coequalisers of reflexive pairs.
Certainly $J$ preserves regular epimorphisms; it also preserves unions of
chains $R \subseteq RR^oR \subseteq RR^oRR^oR \subseteq \dots$, since such
unions are effective in $\C$ and in $\cat{Sh}(\C)$, and
each~\eqref{eq:coveringfamily} is covering. As $J$ also preserves finite
limits, it therefore preserves every part of the construction by which we
calculated the coequaliser of a reflexive pair, and so must preserve the
coequaliser as well. This proves that (3) $\Rightarrow$ (1).

To complete the proof, we now show that (1) $\Rightarrow$ (2). Let $\C$ be a
$\Phi_\mathrm{rc}$-exact category; without loss of generality, a small one. By
Theorem~\ref{thm:embedding}, such a $\C$ has finite limits and coequalisers of
reflexive pairs, and admits a full embedding $J \colon \C \to \E$ into a
Grothendieck topos which preserves them. In particular, $\C$ has, and $J$
preserves, coequalisers of equivalence relations, and so we deduce as in
Section~\ref{subsec:barrexact} that $\C$ is Barr-exact. It remains to show that
the chain of subobjects $R \subseteq RR^oR \subseteq \cdots$ associated to any
reflexive relation $(d,c) \colon R \rightrightarrows X$ in $\C$ admits a stable
colimit. Let $R^\ast \rightrightarrows X$ be the kernel-pair of the coequaliser
of $(d,c)$; we have $R \subseteq RR^oR \subseteq \dots \subseteq R^\ast$ as
subobjects of $X \times X$, and we claim that these inclusions exhibit $R^\ast$
as the desired stable colimit. Now $J(R^\ast)$ is the kernel-pair of the
coequaliser of $(Jd, Jc)$; but because $\E$ satisfies the conditions of (2),
the construction with which we began this proof shows that $J(R^\ast)$ is also
the stable colimit of $JR \subseteq (JR)(JR)^o(JR) \subseteq \cdots$; whence,
since $J$ is fully faithful and lex, $R^\ast$ is the stable colimit of $R
\subseteq RR^oR \subseteq \cdots$ as desired.
\end{proof}

\section{The case of a general $\Phi$, when $\V = \cat{Set}$}\label{sec:4}
In each of the examples of the previous section, we derived elementary
descriptions of particular notions of $\Phi$-exactness in an essentially
\emph{ad hoc} fashion. In this section, we show that---still in the case $\V =
\cat{Set}$---we may give an elementary description which is valid for an
arbitrary class of lex-weights $\Phi$. The key idea needed is Anders Kock's
notion of \emph{postulatedness}. Given a finitely complete $\C$ and a topology
$j$ on it, Kock defines in~\cite{Kock1989Postulated} what it means for a cocone
in $\C$ to be \emph{postulated} with respect to $j$. If $\C$ is small, then the
postulatedness of a cocone is equivalent to its being sent to a colimit by the
functor $\C \to \cat{Sh}_j(\C)$. The relevance this has for us is as follows.
Given $\C$ a small, lex, and $\Phi$-lex-cocomplete category, if
$\Phi$-lex-colimit cocones are postulated with respect to some
\emph{subcanonical} topology on $\C$, then there is a full embedding of $\C$
into a Grothendieck topos via a functor preserving finite limits and
$\Phi$-lex-colimits; whence $\C$ is $\Phi$-exact. Conversely, if $\C$ is
$\Phi$-exact then by Theorem~\ref{thm:embedding}, it admits a full,
$\Phi$-exact, embedding into a Grothendieck topos. In
Section~\ref{sec:relative} below, we will see that this embedding may be taken
to be of the form $\C \to \cat{Sh}_j(\C)$ for some topology $j$ on $\C$; but
now this $j$ must be subcanonical, and all $\Phi$-lex-colimits postulated with
respect to it. Thus for small $\C$, $\Phi$-exactness is equivalent to the
postulatedness of $\Phi$-lex-colimit cones with respect to some subcanonical
topology on $\C$. In fact, this equivalence remains valid even when $\C$ is no
longer small; we now give the details of this argument, including a
reconstruction of those aspects of Kock's theory which will be necessary for
our development.

We begin by giving our formulation of postulatedness, which diverges from
Kock's in two aspects. The first has been anticipated above: a cocone in $\C$
will be postulated in our sense just when it is postulated in Kock's sense with
respect to some subcanonical topology on $\C$; equivalently, with respect to
the canonical topology (that is, the largest subcanonical topology) on $\C$.
The second divergence is one of presentation: we are able to give a more
compact definition because we are using the language of weighted colimits. We
will later see how Kock's presentation can be recovered from ours.

Given $\C$ finitely complete, we say that a morphism $f \colon \phi \to \psi$
of $\P \C$ is \emph{final} if it is orthogonal to every representable---in the
sense that any map from $\phi$ to a representable admits a unique extension
along $f$---and \emph{stably final} when all of its pullbacks are final. Note
that, in particular, a map $\phi \to YC$ is final just when it exhibits $C$ as
the colimit $\phi \star 1_\C$. We now say that $f \colon \phi \to \psi$ is
\emph{postulated} if it satisfies the following two conditions:
\begin{enumerate}[(P1)]
\item[(P1)] The image $\im(f) \rightarrowtail \psi$ of $f$ is stably final;
\item[(P2)] The diagonal $\delta \colon \phi \rightarrowtail \phi
    \times_\psi \phi$ is stably final.
\end{enumerate}

If $\C$ is small-exact, then $Y \colon \C \to \P \C$ admits a left exact left
adjoint $L$ and now a morphism of $\P \C$ is final just when it is inverted by
$L$. Since the left adjoint preserves pullbacks, any map which is inverted by
$L$ is in fact stably inverted; so every final morphism is stably final, and
$Lf$ is invertible if and only if $f$ is stably final. In this context, a
morphism $f$ is postulated if and only if both its image and its diagonal are
inverted by $L$, which is to say that it is \emph{$L$-bidense} in the sense
of~\cite[Definition 3.41]{Johnstone1977Topos}. Still in this context, the
$L$-bidense morphisms are in fact precisely those inverted by
$L$---see~\cite[Corollary 3.43]{Johnstone1977Topos}, for example---so that if
$\C$ is small-exact, a morphism of $\P \C$ is postulated if and only if it is
final; this was shown to be the case in Proposition~2.1
of~\cite{Kock1989Postulated}. Yet even if $\C$ is not small-exact, we still
have:
\begin{Prop}\label{prop:postulatedfinal}
(c.f.~\cite[Proposition 1.1]{Kock1989Postulated}). If $\C$ is finitely
complete, then any postulated morphism in $\P \C$ is stably final.
\end{Prop}
\begin{proof}
Observe that postulated morphisms are stable under pullback, since images and
diagonals are so; hence it is enough to show that any postulated morphism is
final. Given the postulated $f$, form its kernel-pair, its image and the
diagonal of the kernel-pair as in
\begin{equation*}
   \cd[@C-0.1em]{ \phi \ar[r]^-{\delta} & \phi \times_\psi \phi \ar@<4pt>[r]^-d \ar@<-4pt>[r]_-c & \phi \ar[r]^-e & \im(f) \ar[r]^-m & \psi\rlap{ .}}
\end{equation*}
Now $m$ is final by (P1); we must show that $e$ is too, which is to say that
every $g \colon \phi \to YE$ admits a unique extension along $e$. Since $e$ is
the coequaliser of the kernel-pair of $f$, this will happen just if $gd = gc$;
but $gd\delta = g = gc\delta$ and so $gd = gc$ since $\delta$ is final by (P2).
Thus $g$ extends along $e$; the uniqueness is forced since $e$ is epimorphic.
\end{proof}
Thus the force of the discussion preceding this proposition is that for a
small-exact category $\C$, every final morphism in $\P \C$ is postulated. We
now consider the extent to which this remains true on passing from small-exact
categories to $\Phi$-exact ones. First we need a preparatory result.
\begin{Lemma}\label{lem:stablyfinal}
A morphism $f \colon \phi \to \psi$ of $\P \C$ is stably final just when every
pullback of it along a map with representable domain is final.
\end{Lemma}
\begin{proof}
Suppose given some $g \colon \psi' \to \psi$; we are to show that the pullback
$f' \colon \phi' \to \psi'$ of $f$ along $g$ is final. Let $(q_i \colon YC_i
\to \psi' \mid i \in \I)$ exhibit $\psi'$ as a (conical) colimit of
representables. For each $i$, the pullback $f'_i \colon \phi'_i \to YC_i$ of
$f'$ along $q_i$ is a pullback of $f$ along $gq_i$, so final by assumption. As
colimits in $\P \C$ are stable under pullback, $\phi'$ is the colimit of the
$\phi'_i$'s, and hence $f'$ is the colimit in $(\P \C)^\mathbf 2$ of the final
$f'_i$'s, and so itself final, since final maps, being defined by an
orthogonality property, are stable under colimits.
\end{proof}
\begin{Prop}\label{prop:finalispostulated}
If $\Phi$ is a class of lex-weights, and $\C$ a $\Phi$-exact category, then
each final morphism of $\P \C$ lying in $\app \Phi \C$ is postulated.
\end{Prop}
\begin{proof}
As $\C$ is $\Phi$-exact, $W \colon \C \to \app \Phi \C$ admits a left exact
left adjoint $L$, and as above, a morphism of $\app \Phi \C$ is final in $\P
\C$ just when it is inverted by $L$. Since $L$ preserves pullbacks, if $f
\colon \phi \to \psi$ is final and lies in $\app \Phi \C$, then any pullback of
it along a map $YC \to \psi$ is again final, since the representables lie in
$\app \Phi \C$. So by Lemma~\ref{lem:stablyfinal}, $f$ is stably final in $\P
\C$, and it follows that $\im(f) \rightarrowtail \psi$ is stably final, since the
image of any final map is easily shown to be final, and image factorisations in
$\P \C$ are stable under pullback. This verifies (P1) for $f$; as for (P2),
observe that the diagonal $\delta \colon \phi \rightarrowtail \phi \times_\psi
\phi$ lies in $\app \Phi \C$, and is sent by $L$ to the diagonal of the
kernel-pair of $Lf$, which is invertible since $Lf$ is. Thus $\delta$ is final
and lies in $\app \Phi \C$, and so arguing as before, is stably final.
\end{proof}
We may now give the promised correspondence between $\Phi$-exactness and the
postulatedness of $\Phi$-lex-colimits. Given $\Phi$ a class of lex-weights, and
$\C$ a finitely complete and $\Phi$-lex-cocomplete category, by a
\emph{$\Phi$-lex-colimit} morphism in $\P \C$, we mean a final morphism of the
form $\Lan_D(\phi) \to Y(\phi \star D)$ for some $\phi \in \Phi[\K]$ and lex $D
\colon \K \to \C$; and by saying that $\Phi$-lex-colimits are postulated in
$\C$, we mean to say that every such $\Phi$-lex-colimit morphism is postulated.
\begin{Thm}\label{prop:phiexactsetchar}
Let $\Phi$ be a class of lex-weights. Then the finitely complete and
$\Phi$-lex-cocomplete $\C$ is $\Phi$-exact if and only if $\Phi$-lex-colimits
are postulated in $\C$.
\end{Thm}
\begin{proof}
If $\C$ is $\Phi$-exact, then every $\Phi$-lex-colimit morphism, being final
and lying in $\app \Phi \C$, is postulated by
Proposition~\ref{prop:finalispostulated}. Conversely, suppose that each
$\Phi$-lex-colimit morphism in $\P \C$ is postulated; we will show that $\C$ is
$\Phi$-exact. By Propositions~\ref{prop:smallsubphiexact}
and~\ref{prop:setsingletonexact}, we may assume that $\C$ is small, and now we
define a topology on $\C$ as follows. For each $\Phi$-lex-colimit morphism $f
\colon \phi \to YC$ in $\P \C$, we declare that its image $\im(f)
\rightarrowtail YC$ should be a covering sieve, and that for each pair $h, k
\colon YD \rightrightarrows \phi$ with $fh = fk$, their equaliser $\theta
\rightarrowtail YD$ should be a covering sieve. Each $\im(f) \rightarrowtail
YC$ is stably final by (P1), whilst each $\theta \rightarrowtail YD$ is stably
final by (P2), being the pullback of $\delta \colon \phi \rightarrowtail \phi
\times_{YC} \phi$ along some $(h,k) \colon YD \to \phi \times_{YC} \phi$. Hence
these sieves generate a subcanonical topology on $\C$, and we have a full, lex
embedding $J \colon \C \to \cat{Sh}(\C)$. To complete the proof, it is enough
to show that $J$ preserves $\Phi$-lex-colimits; equivalently, that every sheaf
sends $\Phi$-lex-colimits in $\C$ to limits in $\cat{Set}$; equivalently, that
every sheaf $F$ is orthogonal in $\P \C$ to every $\Phi$-lex-colimit morphism
$f \colon \phi \to YC$. Fixing $F$ and $f$, and arguing as in
Proposition~\ref{prop:postulatedfinal}, it is enough to show that $F$ is
orthogonal to $m \colon \im(f) \rightarrowtail YC$ and to $\delta \colon \phi
\rightarrowtail \phi \times_{YC} \phi$. Certainly $F$ is orthogonal to $m$,
since $m$ is covering and $F$ a sheaf; as for $\delta$, it suffices, arguing
now as in Lemma~\ref{lem:stablyfinal}, to demonstrate $F$'s orthogonality to
$g^\ast(\delta)$ for every $D \in \C$ and $g \colon YD \to \phi \times_{YC}
\phi$. But to give $g$ is equally well to give $h, k \colon YD
\rightrightarrows \phi$ satisfying $fh = fk$, and now $g^\ast(\delta)$ is
equally well the equaliser of $h$ and $k$, and so a covering sieve; to which
$F$, being a sheaf, is orthogonal.
\end{proof}
We now explain how our definition of postulatedness relates to Kock's. Suppose
given a finitely complete $\C$ and a map $f \colon \phi \to YC$ in $\P \C$. We
will describe in elementary terms what it means for $f$ to be postulated, doing
so with respect to some presentation of $\phi$ as a coequaliser
\begin{equation}\label{eq:presentation}
    \cd{\textstyle\sum_{i \in I} YA_i \ar@<3pt>[r]^-{s} \ar@<-3pt>[r]_-{t} & \textstyle\sum_{j \in J} YB_j \ar@{->>}[r]^-q & \phi\rlap{ .}}
\end{equation}
Observe that to give $s$ and $t$ is equally well to give functions $\sigma,
\tau \colon I \rightrightarrows J$ and families of maps $(s_i \colon A_i \to
B_{\sigma i} \mid i \in I)$ and $(t_i \colon  A_i \to B_{\tau i} \mid i \in I)$
and that to give $q$ is equally well to give a family of maps $(q_j \colon YB_j
\to \phi \mid j \in J)$ with $q_{\sigma i}.Ys_i = q_{\tau i}.Yt_i$ for each $i
\in I$. Moreover, as $q$ is the coequaliser of $s$ and $t$, to give $f \colon
\phi \to YC$ is equally well to give a family of maps $(r_j \colon B_j \to C
\mid j \in J)$ such that $r_{\sigma i}. s_i = r_{\tau i}. t_i$ for each $i \in
I$. Given now $j, k \in J$, we define a \emph{zig-zag} from $j$ to $k$ to be a
diagram
\begin{equation}\label{eq:zigzag}
    \cd[@!C@C-1em]{ &
    A_{i_1} \ar[dl]_{f_1} \ar[dr]^{g_1} & &
    A_{i_2} \ar[dl]_{f_2} \ar[dr]^{g_2} & &
    & A_{i_n} \ar[dl]_{f_n} \ar[dr]^{g_n}\\
    B_{j_0 = j} & &
    B_{j_1} & &
    {} \ar@{}[r]|{\cdots} & & &
    B_{j_n = k}
}
\end{equation}
where each $i_m \in I$, each $j_m \in J$, and for each $1 \leqslant m \leqslant
n$, either $f_m = s_{i_m}$ and $g_m = t_{i_m}$, or $f_m = t_{i_m}$ and $g_m =
s_{i_m}$. We write $ZZ(j, k)$ for the set of zig-zags from $j$ to $k$. To each
zig-zag $z \in ZZ(j,k)$, we may associate the span $a_z \colon B_j \leftarrow
L_z \rightarrow B_k \colon b_z$ obtained by composing together the spans
appearing in $z$; and now, since $r_{\sigma i}. s_i = r_{\tau i}. t_i$ for each
$i \in I$, also $r_j . a_z = r_k . b_z$, and so there is an induced $\ell_z =
(a_z, b_z) \colon L_z \to B_j \times_C B_k$.
\begin{Prop}\label{prop:postulatedkock}
The morphism $f \colon \phi \to YC$ of $\P \C$ is postulated if and only if:
\begin{enumerate}[(P1')]
\item[(P1')] The family $(r_j \colon B_j \to C \mid j \in J)$ is stably
    effective-epimorphic in $\C$;
\item[(P2')] For all $j, k \in J$, the family $(\ell_z \colon L_z \to B_j
    \times_C B_k \mid z \in ZZ(j,k))$ is stably effective-epimorphic in
    $\C$.
\end{enumerate}
\end{Prop}
Recall that a family of maps $(f_i \colon U_i \to V)$ is
\emph{effective-epimorphic} if it exhibits $V$ as the colimit of the sieve
generated by the $f_i$'s, and is \emph{stably effective-epimorphic} if every
pullback of it along a map $V' \to V$ is effective-epimorphic. The stably
effective-epimorphic families are the covering families for the canonical
topology on $\C$---the largest topology for which each representable functor is
a sheaf---and so, comparing this result with~\cite[Section
1]{Kock1989Postulated}, we deduce as claimed that postulatedness in our sense
coincides with postulatedness in the sense of~\cite{Kock1989Postulated} with
respect to the canonical topology.
\begin{proof}
We show first that (P1) $\Leftrightarrow$ (P1'). Since $q \colon \sum_{j \in J}
YB_j \twoheadrightarrow \phi$ is epimorphic, the images of $f$ and  $fq$
coincide. But since $fq = \spn{Yr_j \mid j \in J} \colon \sum_j YB_j \to YC$,
the image of the latter is the sieve on $C$ generated by the family $(r_j
\colon B_j \to C \mid j \in J)$; so by Lemma~\ref{lem:stablyfinal}, to say that
the image of $f$ is stably final, which is (P1), is equally well to say that
$(r_j \mid j \in J)$ is stably effective-epimorphic, which is (P1').

We now show that (P2) $\Leftrightarrow$ (P2'). First we characterise the sieve
generated by the family $(\ell_z \mid z \in ZZ(j,k))$. By definition, a
morphism $(g,h) \colon X \to B_j \times_C B_k$ lies in this sieve just when it
factorises through some $\ell_z$; that is, just when there is a zig-zag of the
form~\eqref{eq:zigzag}, and an extension of the pair $(g,h)$ to a cone over
this zig-zag. But by virtue of the way that coequalisers are computed in
$\cat{Set}$, this is equally well to say that, on considering the coequaliser
\begin{equation*}
    \cd{\textstyle\sum_{i \in I} \C(X,A_i) \ar@<3pt>[r]^-{s_X} \ar@<-3pt>[r]_-{t_X} & \textstyle\sum_{j \in J} \C(X,B_j) \ar@{->>}[r]^-{q_X} & \phi(X)\rlap{ ,}}
\end{equation*}
the elements $(j,g)$ and $(k,h)$ of the central set have the same image under
$q_X$; which is equally well to say that the map $Y(g, h) \colon YX \to Y(B_j
\times_C B_k)$ factors through the subobject $\theta_{j, k} \colon YB_j
\times_\phi YB_k \rightarrowtail Y(B_j \times_C B_k)$ induced by the universal
property of pullback in the diagram
\begin{equation*}
    \cd{
        YB_j \times_\phi YB_k \ar@{.>}[dr]^{\theta_{j,k}} \ar@/^1em/[drr]^{\pi_2} \ar@/_1em/[ddr]_{\pi_1} \\
        & Y(B_j \times_C B_k) \ar[r]^-{Y\pi_1} \ar[d]_{Y\pi_2} & YB_j \ar[d]^{Yr_j} \\
        & YB_k \ar[r]_{Yr_k} & YC\rlap{ .}
    }
\end{equation*}
We have thus shown that $\theta_{j,k}$ is the image of $(\ell_z \mid z \in
ZZ(j,k))$; and so by Lemma~\ref{lem:stablyfinal}, to say that (P2') holds is to
say that $\theta_{j,k}$ is stably final for all $j,k \in J$. We now show that
this latter condition is equivalent to (P2); that is, to $\delta \colon \phi
\to \phi \times_{YC} \phi$ being stably final. Now for each $j, k \in J$, the
map $\theta_{j,k}$ is the pullback of $\delta$ along $q_j \times_{YC} q_k$, so
that if $\delta$ is stably final, then each $\theta_{j,k}$ is too. If
conversely each $\theta_{j,k}$ is stably final, then by
Lemma~\ref{lem:stablyfinal}, $\delta$ will be stably final as soon as every
pullback of it along a map $(h,k) \colon YD \to \phi \times_{YC} \phi$ is
final. For any such map we have, since the family $(q_j \mid j \in J)$ is
jointly epimorphic, factorisations $h = q_j u$ and $k = q_k v$ for some $j, k
\in J$ and $(u, v) \colon YD \to Y(B_j \times_C B_k)$; whence the pullback
$(h,k)^\ast(\delta)$ is in fact a pullback of $\theta_{j,k}$, and so indeed
final.
\end{proof}
We give one final formulation of postulatedness; this is the most useful in
practice.
\begin{Prop}\label{prop:postulateduseful}
The morphism $f \colon \phi \to YC$ of $\P \C$ is postulated if and only if it
is stably final and satisfies (P2').
\end{Prop}
\begin{proof}
If $f$ is postulated, then it is stably final by
Proposition~\ref{prop:postulatedfinal}, and satisfies (P2') by
Proposition~\ref{prop:postulatedkock}. Conversely, if $f$ is stably final, then
so is its image, since image factorisations are stable under pullback in $\P
\C$, which verifies (P1); now if it also verifies (P2') then it is postulated
by Proposition~\ref{prop:postulatedkock}.
\end{proof}
We conclude this section with an application of the preceding results; we will
use them to reconstruct the characterisation of adhesive categories given in
Section~\ref{subsec:adhesive}. Recall that $\Phi_\mathrm{adh}$ is the class of
lex-weights such that $\Phi_\mathrm{adh}$-lex-cocomplete categories are those
admitting pushouts along monomorphisms. By Theorem~\ref{prop:phiexactsetchar}
and Proposition~\ref{prop:postulateduseful}, such a category is
$\Phi_\mathrm{adh}$-exact if and only if pushouts along monomorphisms are
stable under pullback, with the corresponding final morphisms of $\P \C$
satisfying condition (P2'). To analyse this latter condition further, let
\begin{equation}\label{eq:thepushoutsquare}
    \cd{
        C \ar[r]^f \ar@{ >->}[d]_m & B \ar[d]^n \\
        A \ar[r]_g  & D
    }
\end{equation}
be a typical pushout along a monomorphism in $\C$, and let $q \colon \phi \to
YD$ be the corresponding final morphism in $\P \C$. We may present $\phi$ as
the coequaliser of the pair $(\iota_1.Ym, \iota_2.Yf) \colon YC
\rightrightarrows YA + YB$, and with respect to this presentation, condition
(P2') for the postulatedness of $q$ breaks up into four clauses; we now
consider these in turn.
\begin{enumerate}[(i)]
\item \emph{The family $(\ell_z \colon L_z \to B \times_D B \mid z \in
    ZZ(B,B))$ should be stably effective-epimorphic.} Every zig-zag in
    $ZZ(B,B)$ is given by zero or more copies of the zig-zag
\begin{equation}\label{eq:thezigzag}
    \cd{
    & C \ar[dl]_f \ar@{ >->}[dr]^m & & C \ar@{ >->}[dl]_m \ar[dr]^f \\
    B & & A & & B
    }
\end{equation}
    placed side by side. From the case $n = 0$, we see that the diagonal
    $\delta \colon B \to B \times_D B$ is in the family $(\ell_z)$. But
    since $m$ is monic, the span composite of the displayed zig-zag has
    both projections onto $B$ equal, and it follows that the span composite
    of every $z \in ZZ(B,B)$ has both projections onto $B$ equal: in other
    words, that every $\ell_z \colon L_z \to B \times_D B$ factors through
    $\delta$. Thus to say that the family $(\ell_z)$ is stably
    effective-epimorphic is equally well to say that the singleton family
    $\delta$ is so; but since $\delta$ is monomorphic, this is equivalent
    to saying that it is invertible, or in other words, that $n$ is monic.
\vskip0.5\baselineskip
\item \emph{The family $(\ell_z \colon L_z \to A \times_D B \mid z \in
    ZZ(A,B))$ should be stably effective-epimorphic.} Every zig-zag in
    $ZZ(A,B)$ is given by zero or more copies of the
    zig-zag~\eqref{eq:thezigzag} placed next to the span $m \colon A
    \leftarrowtail C \to B \colon f$. So in particular, $(m,f) \colon C \to
    A \times_D B$ is in the family $(\ell_z)$, and arguing as before, any
    other $\ell_z$ must factor through this one. So the stated condition is
    equivalent to the singleton family $(m,f)$ being stably
    effective-epimorphic, and since $(m,f)$ is monic (as $m$ is) this is in
    turn equivalent to $(m,f)$ being invertible; that is, to the
    pushout~\eqref{eq:thepushoutsquare} also being a
    pullback.\vskip0.5\baselineskip
\item \emph{The family $(\ell_z \colon L_z \to B \times_D A \mid z \in
    ZZ(B,A))$ should be stably effective-epimorphic.} This condition is
    clearly equivalent to the previous one.\vskip0.5\baselineskip
\item \emph{The family $(\ell_z \colon L_z \to A \times_D A \mid z \in
    ZZ(A,A))$ should be stably effective-epimorphic.} Every zig-zag in
    $ZZ(A,A)$ is given by zero or more copies of
\begin{equation*}
    \cd{
    & C \ar[dr]^f \ar@{ >->}[dl]_m & & C \ar@{ >->}[dr]^m \ar[dl]_f \\
    A & & B & & A
    }
\end{equation*}
    placed side by side. From the cases $n = 0, 1$, we see that $\delta
    \colon A \to A \times_D A$ and $m \times_n m \colon C \times_B C \to A
    \times_D A$ are in the family $(\ell_z)$, and arguing as before, any
    other $\ell_z$ must factor through $m \times_n m$. Thus the stated
    condition is equally that the pair of maps $\delta$ and $m \times_n m$
    should comprise a stably effective-epimorphic family. Since both are
    monomorphic, this is equally well to say that they are the stable
    pushout of their intersection: but this intersection is easily seen to
    be $C$, and so the condition is that
    \begin{equation}\label{eq:kerneldiag}
        \cd{
            C \ar[r]^-\delta \ar[d]_m & C \times_B C \ar[d]^{m \times_n m} \\
            A \ar[r]_-\delta & A \times_D A
        }
    \end{equation}
    should be a stable pushout. In fact it is enough merely that it is a
    pushout, as then it is a pushout along a monomorphism and so stable by
    assumption.
\end{enumerate}
In conclusion, we see that the finitely complete category $\C$ is
$\Phi_\mathrm{adh}$-exact just when pushouts along monomorphisms exist, are
stable, and are pullbacks, when monomorphisms are stable under pushout, and
when finally for every pushout square~\eqref{eq:thepushoutsquare}, the
corresponding square~\eqref{eq:kerneldiag} is also a pushout. Now we saw in the
proof of Proposition~\ref{prop:adhchar} that pushouts of monomorphisms are
monomorphisms provided that such pushouts are stable, so that this condition
can be omitted; moreover, Lemma 3.2 of~\cite{Lack2011An-embedding} shows that the
condition involving~\eqref{eq:kerneldiag} is also a consequence of the others.
Thus we conclude that $\C$ is $\Phi_\mathrm{adh}$-exact just when pushouts
along monomorphisms are stable and are pullbacks: which is what we proved in
Proposition~\ref{prop:adhchar}.

\section{Relative completions}\label{sec:relative}
In this final section, we return to the development, for a general $\V$, of the
theory of $\Phi$-exactness. Our goal is to describe circumstances under which
it is possible to construct the free $\Psi$-exact completion of a $\Phi$-exact
category. First we need to ascertain the circumstances under which it makes
sense even to speak of the free $\Psi$-exact completion of a $\Phi$-exact
category; to which end, we introduce the following notation.

Given classes of lex-weights $\Phi$ and $\Psi$, we write $\Phi \leqslant \Psi$
to indicate that the forgetful functor $\Psi\text-\cat{EX} \to \cat{LEX}$
factors through $\Phi\text-\cat{EX}$; which is to say that every $\Psi$-exact
category is $\Phi$-exact, and every $\Psi$-exact functor $\Phi$-exact. There
are various ways of characterising this ordering.
\begin{Prop}\label{prop:charorderrelation}
Given classes of lex-weights $\Phi$ and $\Psi$, the following are equivalent:
\begin{enumerate}
\item $\Phi \leqslant \Psi$;
\item $\Phi \subseteq \sat \Psi$;
\item $\sat \Phi \subseteq \sat \Psi$;
\item $\app \Phi \C \subseteq \app \Psi \C$ for all small, finitely
    complete $\C$;
\item $\app \Phi \C \subseteq \app \Psi \C$ for all finitely complete $\C$.
\end{enumerate}
\end{Prop}
\begin{proof}
If (1) holds, then for any finitely complete $\C$, the category $\app \Psi \C$
and inclusion $\app \Psi \C \to \P \C$ are both $\Phi$-exact, since
$\Psi$-exact. Thus $\app \Psi \C$ is closed in $\P \C$ under finite limits and
$\Phi$-lex-colimits, whence $\app \Phi \C \subseteq \app \Psi \C$ by
Proposition~\ref{prop:phicchar}. Thus (1) $\Rightarrow$ (5); and trivially (5)
$\Rightarrow$ (4) $\Rightarrow$ (3) $\Rightarrow$ (2), so it remains to show
(2) $\Rightarrow$ (1). As it is clear from
Propositions~\ref{prop:charphilexcocomp} and~\ref{prop:charphilexmaps} that
$\Psi\text-\cat{EX} = \sat \Psi\text-\cat{EX}$, it is enough to show that if
$\Phi \subseteq \Psi$ then $\Phi \leqslant \Psi$. But if $\Phi \subseteq \Psi$,
then clearly $\app \Phi \C \subseteq \app \Psi \C$ for all finitely complete
$\C$, so that if $\C \to \app \Psi \C$ has a lex left adjoint, then so does $\C
\to \app \Phi \C$. Thus every $\Psi$-exact category is also $\Phi$-exact, and
clearly any $\Psi$-exact functor is $\Phi$-exact, so that $\Phi \leqslant \Psi$
as desired.
\end{proof}
Taking $\Phi = \{\psi\}$ in the above, we immediately deduce the following
result, which can be seen as the analogue, for our theory, of~\cite[Theorem
5.1]{Albert1988The-closure}.
\begin{Cor} If $\Psi$ is a class of
lex-weights, then $\psi \in \Psi^\ast$ if and only if every $\Psi$-exact
category is also $\{\psi\}$-exact, and every $\Psi$-exact functor is
$\{\psi\}$-exact.
\end{Cor}
Whenever $\Phi \leqslant \Psi$, we have a forgetful $2$-functor
$\Psi\text-\cat{EX} \to \Phi\text-\cat{EX}$; and we now investigate the extent
to which this has a left biadjoint. We saw in Corollary~\ref{cor:biadjoint}
that such a biadjoint exists when $\Phi$ is the minimal class of lex-weights,
and $\Psi$ arbitrary; and we next shall consider the other extremal case, in
which $\Psi$ is maximal, and $\Phi$ arbitrary. In other words, we wish to
describe the free small-exact completion of the $\Phi$-exact $\C$.

For reasons of size, we cannot expect always to be able to do this; but we may
do so, at least, whenever $\C$ is small. For such a $\C$, we will construct its
small-exact completion as a suitable lex-reflective subcategory of $[\C^\op,
\V]$, into which $\C$ will embed via the (restricted) Yoneda embedding.
Certainly this embedding will preserve finite limits; we wish it also to
preserve $\Phi$-lex-colimits. But this is equally well, by
Proposition~\ref{prop:charphilexmaps}, to ask that it should preserve all $\sat
\Phi$-lex-colimits; which in turn is equivalent to the requirement that every
$F \colon \C^\op \to \V$ in our subcategory should send $\sat
\Phi$-lex-colimits in $\C$ to limits in $\V$. Let us therefore write $\P_\Phi
\C$ for the full subcategory of $[\C^\op, \V]$ spanned by the functors with
this property, and, recognising that every representable lies in $\P_\Phi \C$,
write $V \colon \C \to \P_\Phi \C$ for the restricted Yoneda embedding. The
first step in showing that this constitutes a small-exact completion of $\C$ is
to prove:

\begin{Prop}\label{prop:PPhilexreflective}
If $\C$ is small and $\Phi$-exact, then $\P_\Phi \C$ is lex-reflective in
$[\C^\op, \V]$, and hence small-exact.
\end{Prop}
In proving this proposition, we will need to make use of a technical result; it
states that the full, replete, lex-reflective subcategories of $[\C^\op, \V]$
form a small, complete lattice, in which infima are given by intersection. In
the unenriched case, this result---or rather a generalisation of it---was
proved by Borceux and Kelly in~\cite{Borceux1987On-locales}; we recall their
proof, indicating its adaptation to the enriched setting, as
Proposition~\ref{prop:lexrefl} below.
\begin{proof} We proceed first under the assumption that $\app \Phi
\C$ is small. In this case, taking $L \colon \app \Phi \C \to \C$ to be a left
exact left adjoint for $W \colon \C \to \app \Phi \C$, we may consider the
following string of adjunctions
\begin{equation*}
    \cd[@C+5em]{
        [\C^\op, \V]
        \ar@<8pt>[r]|-{\Sigma_W \defeq \Lan_{W^\op}}
        \ar@<16pt>@{}[r]|-{\bot}
        \ar@<1pt>@{}[r]|-{\bot}
        \ar@<-14pt>@{}[r]|-{\bot}
        \ar@<-20pt>[r]_-{\Pi_W \defeq \Ran_{W^\op}} &
        [(\app \Phi\C)^\op, \V]\rlap{ .}
        \ar@<8pt>[l]|-{\Delta_W \defeq [W^\op, 1]}
        \ar@<-21pt>[l]_-{\Sigma_L \defeq \Lan_{L^\op}}
    }
\end{equation*}
Each of the functors appearing in it is left exact, the lower three since they
are right adjoints, and $\Sigma_L$ because $L$ is. Since $W$ is fully faithful,
so are $\Sigma_W$ and $\Pi_W$, and their essential images $\Ss$ and $\T$ are both therefore lex-reflective subcategories of $[(\app \Phi \C)^\op, \cat{Set}]$. We shall show that the intersection $\Ss \cap \T$ is also lex-reflective, and that it is equivalent to $\P_\Phi \C$; from this the result  then follows.

Since $\Sigma_W$ is fully faithful, the unit $\eta$ of the adjunction
$\Sigma_W \dashv \Delta_W$ is invertible, and on composing its inverse with the
unit $\nu$ of the adjunction $\Delta_W \dashv \Pi_W$, we obtain a natural
transformation
\begin{equation*}
    \theta \defeq \Sigma_W \xrightarrow{\nu.1} \Pi_W \Delta_W \Sigma_W \xrightarrow{1.\eta^{-1}} \Pi_W\rlap{ .}
\end{equation*}
We claim that $F \in [\C^\op, \V]$ lies in $\P_\Phi \C$ just when $\theta_F$ is
invertible. Indeed, since $L \dashv W$, we have $\Sigma_W \cong [L^\op, 1]$;
whence $(\Sigma_W F)(\phi) \cong F(L\phi) \cong F(\phi \star 1_\C)$. On the
other hand, we have $    (\Pi_W F)(\phi) = \{\app\Phi\C(W, \phi), F\} =
\{[\C^\op, \V](Y, \phi) , F\} \cong \{\phi, F\}$, and it is straightforward to
verify that under these isomorphisms, the map $(\theta_F)_\phi$ is identified
with the canonical comparison map $F(\phi \star 1_\C) \to \{\phi, F\}$. So $F
\in \P_\Phi \C$ just when $\theta_F$ is invertible, as claimed.

As observed above, both $\Sigma_W$ and $\Pi_W$ are fully faithful, and have
left exact left adjoints; consequently, they determine idempotent left exact
monads $S$ and $T$ on $[(\app\Phi\C)^\op, \V]$, whose respective categories of
algebras---denoted by $\Ss$ and $\T$---are isomorphic to the replete images of
$\Sigma_W$ and $\Pi_W$ in $[(\app\Phi \C)^\op, \V]$. Moreover, to say that $F
\in [\C^\op, \V]$ lies in $\P_\Phi \C$ is by the above to say that $\theta_F$,
and hence $\nu_{(\Sigma_W F)} \colon \Sigma_W F \to \Pi_W \Delta_W \Sigma_W F =
T( \Sigma_W F)$ is invertible, which is equally well to say that $\Sigma_W
F$---which necessarily lies in $\Ss$---also lies in $\T$. Conversely, if $G \in
[(\app\Phi\C)^\op, \V]$ lies in $\Ss \cap \T$, then $GW^\op$ must lie in
$\P_\Phi \C$, since the map $\theta_{G W^\op}$ may be decomposed as the
composite
\begin{equation*}
    \Sigma_W \Delta_W G \xrightarrow{\epsilon} G \xrightarrow{\nu} \Pi_W \Delta_W G
\end{equation*}
of the counit of $\Sigma_W \dashv \Delta_W$ with the unit of $\Delta_W \dashv
\Pi_W$ at $G$; but both these maps are invertible, the first by the assumption
that $G \in \Ss$, and the second by the assumption that $G \in \T$.
Consequently, if we can show that $\Ss \cap \T$ is lex-reflective in
$[(\app\Phi\C)^\op, \V]$ via a reflector $\rho \colon 1 \to R$, we can conclude
that $\P_\Phi \C$ is lex-reflective in $[\C^\op, \V]$, via
\begin{equation*}
    1 \xrightarrow \eta \Delta_W \Sigma_W \xrightarrow{1.\rho.1} \Delta_W  R \Sigma_W\ \text.
\end{equation*}
But $\Ss$ and $\T$ are both lex-reflective in $[(\app\Phi\C)^\op, \V]$, and
thus, by Proposition~\ref{prop:lexrefl}, so is $\Ss \cap \T$. This proves that
$\appr \P \Phi \C$ is lex-reflective in $[\C^\op, \V]$ whenever $\app \Phi \C$
is small.

We now drop the assumption on $\app \Phi \C$. To show that $\appr \P \Phi \C$
is still lex-reflective, let us observe that for each $\phi \in \Phi$,
$\app{\{\phi\}} \C$ is certainly small, since $\C$ is, so that each
$\P_{\{\phi\}}\C$ is lex-reflective in $[\C^\op, \V]$ by the case just proved.
Thus by Proposition~\ref{prop:lexrefl}, $\E = \bigcap_\phi \P_{\{\phi\}}\C$ is
also lex-reflective in $[\C^\op, \V]$: we claim that it is in fact $\P_\Phi\C$,
which will complete the proof. Clearly $\P_\Phi \C \subseteq \E$; for the
converse, we must show that each $F \colon \C^\op \to \V$ in $\E$ sends
$\Phi^\ast$-lex-colimits in $\C$ to limits in $\V$. But this is equally well to
ask that the induced functor $Z \colon \C \to \E$ should preserve
$\Phi^\ast$-lex-colimits, which since $\C$ and $\E$ are both $\Phi$-exact, is
equally well to ask that $Z$ should preserve $\Phi$-lex-colimits; which in turn
is to ask that each $F \in \E$ should send $\Phi$-lex-colimits to limits. But
this is just to ask that for each $\phi \in \Phi$, each $F \in \E$ should send
$\{\phi\}$-lex-colimits to limits, which is so because $F \in \E \subseteq
\P_{\{\phi\}} \C$.
\end{proof}
\begin{Cor}\label{prop:requiredembedding}
If $\C$ is small and $\Phi$-exact, then ${V} \colon \C \to \P_\Phi \C$ is a
full, $\Phi$-exact embedding of $\C$ into a $\V$-topos.
\end{Cor}
\begin{proof}
In light of the preceding proposition, it suffices to show that $V$ is a
$\Phi$-exact functor. Certainly it preserves finite limits; as for
$\Phi$-lex-colimits, we must show that if $\phi \in \Phi[\K]$ and $D \colon \K
\to \C$, then ${V}$ preserves the colimit $\phi \star D$: for which we
calculate that
\begin{align*}
    & \ \ \ \ \P_\Phi \C({V}(\phi \star D), F) \cong F(\phi \star D) \cong \{\phi, FD^\op\}
\\ & \cong [\K^\op, \V](\phi, FD^\op) \cong [\K^\op, \V](\phi, \P_\Phi \C({V}D\thg,
F))\rlap{ .} \qedhere
\end{align*}
\end{proof}
Given the preceding results, it is now an essentially standard argument to
prove that:
\begin{Thm}\label{prop:smallexactcompletion}
If $\C$ is small and $\Phi$-exact, then ${V} \colon \C \to \P_\Phi \C$ provides
a bireflection of $\C$ along the forgetful $2$-functor $\infty\text-\cat{EX}
\to \Phi\text-\cat{EX}$.
\end{Thm}
Here we write $\infty\text-\cat{EX}$ for the $2$-category of small-exact
categories, small-exact functors and arbitrary natural transformations.
\begin{proof}
By the preceding two results, $\appr \P \Phi \C$ is small-exact and $V$ is
$\Phi$-exact; and so composition with ${V}$ induces, for any small-exact
category $\D$, a functor
\begin{equation*}
    {V}^\ast \colon \infty\text-\cat{EX}(\P_\Phi \C, \D) \to \Phi\text-\cat{EX}(\C, \D)
\end{equation*}
which we are to show is an equivalence. As is typical, we do this by exhibiting
as its pseudoinverse the functor which on objects sends $F \colon \C \to \D$ to
$\Lan_{V} F \colon \P_\Phi \C \to \D$. First we show that this is well-defined;
that is, that $\Lan_{V} F$ is a small-exact functor whenever $F$ is
$\Phi$-exact. First we note that $\Lan_{V} F$ preserves finite limits: indeed,
$\Lan_Y F \colon \P \C \to \D$ preserves finite limits since $\D$ is
small-exact, and now because the inclusion $I \colon \P_\Phi \C \to \P \C$ is
fully faithful, we have $\Lan_{V} F \cong (\Lan_Y F) . I$, the composite of two
finite-limit preserving functors, and so itself finite-limit preserving. It
remains to show that $\Lan_{V} F$ is cocontinuous. For this, observe that
$\Lan_Y F$ has as right adjoint the singular functor $\tilde F \colon \D \to \P
\C$ sending $D$ to $\D(F\thg, D)$; so that if we can show that $\tilde F$
factors through $\P_\Phi \C$ as $\tilde F = IR$, say, then we will have
$\Lan_{V} F \dashv R$ and so $\Lan_{V} F$ cocontinuous. But since $F$ preserves
$\Phi$-lex-colimits, $\D(F\thg, D)$ will certainly send them to limits in $\V$
for each $D$, and so we have the desired factorisation of $\tilde F$ through
$\P_\Phi \C$. We have therefore shown that $\Lan_{V}$ is a functor
$\Phi\text-\cat{EX}(\C, \D) \to \infty\text-\cat{EX}(\P_\Phi \C, \D)$, and it
remains to show that it is pseudoinverse to ${V}^\ast$. But since ${V}$ is
fully faithful, we have ${V}^\ast.\Lan_{V} \cong 1$; and since ${V}$ is dense,
we have $\Lan_{V} . {V}^\ast \cong 1$.
\end{proof}
Before continuing, let us use the preceding results to complete the proof of
Proposition~\ref{prop:setsingletonexact}, which we restate here as:
\begin{Prop}\label{prop:setsingletonexact2}
For $\Phi$ a class of lex-weights, a category $\C$ is $\Phi$-exact if and only
if it is $\{\phi\}$-exact for each $\phi \in \Phi$.
\end{Prop}
\begin{proof}
If $\C$ is $\Phi$-exact, then it is $\{\phi\}$-exact for each $\phi \in \Phi$
by Proposition~\ref{prop:charorderrelation}. Suppose conversely that $\C$ is
$\{\phi\}$-exact for each $\phi \in \Phi$. By
Proposition~\ref{prop:smallsubphiexact}, $\C$ will be $\Phi$-exact if we can
show every small, full, finite-limit- and $\Phi$-lex-colimit-closed subcategory
$\D \subseteq \C$ to be $\Phi$-exact. Fix such a $\D$. Now for each $\phi \in
\Phi$, $\D$ is $\{\phi\}$-exact, since $\C$ is, and so $\P_{\{\phi\}} \D$ is
lex-reflective in $[\D^\op, \V]$. As in the proof of
Proposition~\ref{prop:PPhilexreflective}, we therefore have $\P_\Phi(\D) =
\bigcap_\phi \P_{\{\phi\}}(\D)$ also lex-reflective in $[\D^\op, \V]$, so that
by Theorem~\ref{thm:embedding}, $\D$ is $\Phi$-exact as required.
\end{proof}
We now give our final result, which, for an arbitrary pair of classes $\Phi
\leqslant \Psi$, describes the $\Psi$-exact completion of the small
$\Phi$-exact $\C$. Given such a $\C$, we write $\appr \Psi \Phi \C$ for the
closure of $\C$ in $\P_\Phi \C$ under finite limits and $\Psi$-lex-colimits,
and
\begin{equation*}
    V = \C \xrightarrow{Z} \appr \Psi \Phi \C \xrightarrow{K} \P_\Phi \C
\end{equation*}
for the factorisation of $V$ this induces.
\begin{Thm}\label{thm:relativecompletion}
Let $\Phi \leqslant \Psi$ be classes of lex-weights, and let $\C$ be small and
$\Phi$-exact. Now $Z \colon \C \to \appr \Psi \Phi \C$ provides a bireflection
of $\C$ along the forgetful $2$-functor $\Psi\text-\cat{EX} \to
\Phi\text-\cat{EX}$.
\end{Thm}
For instance, if $\V = \cat{Set}$, $\Phi = \Phi_\mathrm{reg}$ and $\Psi =
\Phi_\mathrm{ex}$, then our $\appr \Psi \Phi \C$ is what is typically referred
to as the ex/reg completion of the regular category $\C$; that it can be
constructed in the above manner, by closing the representables in the topos of
sheaves for the regular topology under finite limits and coequalisers of
equivalence relations, was shown in~\cite{Lack1999A-note}.
\begin{proof}
First observe that $\appr \Psi \Phi \C$ is $\Psi$-lex-cocomplete, and $K$ a
full, $\Psi$-lex-cocontinuous, embedding of it into a $\Psi$-exact category;
whence, by Theorem~\ref{thm:embedding}, $\appr \Psi \Phi \C$ is $\Psi$-exact.
Moreover, the embedding $Z \colon \C \to \appr \Psi \Phi \C$ preserves finite
limits and $\Phi$-lex-colimits, since $V$ preserves, and $K$ reflects them.
Thus for any $\Psi$-exact category $\D$, composition with $Z$ induces a functor
\begin{equation*}
    Z^\ast \colon \Psi\text-\cat{EX}(\appr \Psi \Phi \C, \D) \to \Phi\text-\cat{EX}(\C, \D)\ \text,
\end{equation*}
which we are to show is an equivalence. As before, we shall do so by showing
that a suitable pseudoinverse is given by left Kan extension along $Z$.

We prove the result first under the assumptions that $\D$ is small, and that
$\Phi$ and $\Psi$ are both small classes of lex-weights; recall from
Section~\ref{sec:embedding} that this means that $\app \Phi \K$ and $\app \Psi
\K$ will be small whenever $\K$ is. Under these circumstances, with both $\C$
and $\D$ being small, we may form their small-exact completions $V \colon \C
\to \appr \P \Phi \C$ and $U \colon \D \to \appr \P \Psi \D$. Now for any
$\Phi$-exact $F \colon \C \to \D$, the composite $UF \colon \C \to \appr \P
\Psi \D$ is again $\Phi$-exact; it also has small-exact codomain, and so by
Theorem~\ref{prop:smallexactcompletion}, $\Lan_V(UF) \colon \appr \P \Phi \C
\to \appr \P \Psi \D$ exists and is small-exact. Since $\Lan_V(UF)$ preserves
finite limits and $\Psi$-lex-colimits, and maps $\C$ into the replete image of
$\D$ in $\appr \P \Psi \D$, it must map $\appr \Psi \Phi \C$ into the closure
of $\D$ in $\appr \P \Psi \D$ under finite limits and $\Psi$-lex-colimits. But
this is again just the replete image of $\D$ in $\appr \P \Psi \D$, and so
there is a factorisation:
\begin{equation*}
    \cd[@!C]{
        \C \ar[d]_F \ar[r]^Z \twocong{dr} &
        \appr \Psi \Phi \C \twocong{dr} \ar[r]^{K} \ar@{.>}[d]|{\bar F} &
        \appr \P \Phi \C \ar[d]^{\Lan_V(UF)} \\
        \D \ar@{=}[r] & \D \ar[r]_U & \appr \P \Psi \D\rlap{ .}
    }
\end{equation*}
Now as $K$ is fully faithful, we have $U\bar F \cong \Lan_V(UF).K \cong
\Lan_Z(UF)$; but since $U$ is fully faithful, it reflects Kan extensions,
whence $\bar  F$ is a left Kan extension of $F$ along $Z$. Moreover, as
$\Lan_V(UF)$ and $K$ are $\Psi$-exact, and $U$ is full and faithful, it follows
that $\bar F$ is $\Psi$-exact. Thus we have shown that every $\Phi$-exact $F
\colon \C \to \D$ admits a $\Psi$-exact left Kan extension along $Z$, so
determining a functor $\Phi\text-\cat{EX}(\C, \D) \to \Psi\text-\cat{EX}(\appr
\Psi \Phi \C, \D)$; it remains to show that this functor is pseudoinverse to
$Z^\ast$. Certainly we have $Z^\ast.\Lan_Z \cong 1$, since $Z$ is fully
faithful; on the other hand, for any $\Psi$-exact $G \colon \appr \Psi \Phi \C
\to \D$, the collection of $\phi \in \appr \Psi \Phi \C$ at which the component
of the canonical $\Lan_Z(GZ) \to G$ is invertible contains the representables
(as $Z^\ast.\Lan_Z.Z^\ast \cong Z^\ast$) and is closed under finite limits and
$\Psi$-lex-colimits (since both $G$ and $\Lan_Z(GZ)$ are $\Psi$-exact), and so
must be all of $\appr \Psi \Phi \C$; thus $\Lan_Z.Z^\ast \cong 1$ as required.
This completes the proof under the assumptions that $\Phi$, $\Psi$ and $\D$ are
all small.

We now drop the assumption that $\D$ is small. To complete the proof in this
case, it is enough to show that every $\Phi$-exact $F \colon \C \to \D$ admits
a $\Psi$-exact left Kan extension along $Z$, as then we may conclude the
argument as before. To this end, let $\E$ be the closure of the image of $F$ in
$\D$ under finite limits and $\Psi$-lex-colimits, and
\begin{equation*}
    F = \C \xrightarrow G \E \xrightarrow H \D
\end{equation*}
the factorisation so induced. Now $\E$ is $\Psi$-lex-cocomplete, and $H$ is
$\Psi$-lex-cocontinuous, so that by Theorem~\ref{thm:embedding}, both $\E$ and
$H$ are in fact $\Psi$-exact. Moreover, $G$ is $\Phi$-exact---since $F$ is
$\Phi$-exact and $H$ fully faithful---and $\E$ is small, being the closure of a
small subcategory under a small class of lex-weights; so that by the case just
proved, $\Lan_Z G$ exists and is $\Psi$-exact. Consequently, $H.\Lan_Z G$ is
$\Psi$-exact, and we will be done if we can show that it is in fact $\Lan_Z F$.
Equivalently, we may show that $H$ preserves $\Lan_Z G$; equivalently, that for
each $\psi \in \appr \Psi \Phi \C$, the colimit $\psi \star G$ in $\E$ is
preserved by $H$; or equivalently, that for each $\psi  \in \appr \Psi \Phi \C$
and $X \in \D$, the canonical morphism
\begin{equation}\label{eq:kanpropeq}
    \D(H(\psi \star G), X) \to [\C^\op, \V](\psi, \D(F\thg, X))
\end{equation}
is invertible in $\V$. To do this last, we let $\E'$ be the closure of $\E \cup
\{X\}$ in $\D$ under finite limits and $\Psi$-lex-colimits. Arguing as before,
$\E'$ is small and $\Psi$-exact, and the inclusion $K \colon \E \to \E'$ is
$\Psi$-exact. Thus $K.\Lan_Z G$ is also $\Psi$-exact, and so, having small
codomain, is $\Lan_Z(KG)$ by the case just proved. Hence the canonical morphism
\begin{equation*}
    \E'(K(\psi \star G), X) \to [\C^\op, \V](\psi, \E'(KG\thg, X))
\end{equation*}
is invertible in $\V$; but this is equally well the
morphism~\eqref{eq:kanpropeq}, since the inclusion $\E' \to \D$ is fully
faithful. Thus $H.\Lan_Z G$ is $\Lan_Z F$ as claimed; which completes the proof
in the case where $\Phi$ and $\Psi$ are both small.

We next drop the assumption that $\Psi$ is small. Under these circumstances, it
is enough to show as before that every $\Phi$-exact $F \colon \C \to \D$ admits
a $\Psi$-exact left Kan extension along $Z$. So consider the collection of
$\psi \in \appr \Psi \Phi \C$ for which there exists a small $\Psi'$ with $\Phi
\leqslant \Psi' \leqslant \Psi$ and $\psi \in \appr {\Psi'} \Phi \C \subseteq
\appr \Psi \Phi \C$. It is easy to see that this collection contains the
representables and is closed under finite limits and $\Psi$-lex-colimits, and
hence is all of $\appr \Psi \Phi \C$. So for every $\psi \in \appr \Psi \Phi
\C$, we choose such a $\Psi'$; now by the case just proved, $F$ admits a
$\Psi'$-exact left Kan extension along $\C \to \appr {\Psi'} \Phi \C$, so that,
in particular, the colimit $\psi \star F$ exists in $\D$. Thus the left Kan
extension $\Lan_Z F \colon \appr \Psi \Phi \C \to \D$ exists, and it remains to
show that it is $\Psi$-exact. To see that it preserves $\Psi$-lex-colimits,
suppose given some $\psi \in \Psi[\K]$ and lex $D \colon \K \to \appr \Psi \Phi
\C$; we must show that $\Lan_Z F$ preserves $\psi \star D$. Choosing a small
$\Phi \leqslant \Psi' \leqslant \Psi$ such that $\psi$ and each $DX$ lie in
$\appr {\Psi'} \Phi \C$, we observe that the left Kan extension of $F$ along
$\C \to \appr {\Psi'} \Phi \C$, being $\Psi'$-exact, will preserve the colimit
$\psi \star D$: from which it follows easily that $\Lan_Z F$ does too. A
corresponding argument shows that $\Lan_Z F$ preserves all finite limits, and
so is indeed $\Psi$-exact; which completes the proof under the assumption that
$\Phi$ is small.

We now drop this final assumption. To complete the proof, we must again show
that each $\Phi$-exact $F \colon \C \to \D$ admits a $\Psi$-exact left Kan
extension along $Z$. Now for each $\phi \in \Phi$, we have $\P_{\{\phi\}} \C$
lex-reflective in $[\C^\op, \V]$; and as in the proof of
Proposition~\ref{prop:PPhilexreflective}, we have in fact that $\appr \P \Phi
\C = \bigcap_{\phi \in \Phi} \P_{\{\phi\}} \C$. But by
Proposition~\ref{prop:lexrefl}, $[\C^\op, \V]$ has only a small set of
lex-reflective subcategories, and so there is some small $\Phi' \subseteq \Phi$
such that $\bigcap_{\phi \in \Phi} \P_{\{\phi\}} \C = \bigcap_{\phi \in \Phi'}
\P_{\{\phi\}} \C$. Thus $\appr \P \Phi \C = \appr \P {\Phi'} \C$, whence also
$\appr \Psi \Phi \C = \appr \Psi {\Phi'} \C$; thus $Z$ is equally well the
embedding $\C \to \appr \Psi {\Phi'} \C$, so that, by the case just proved, any
$\Phi$-exact $F$, being \emph{a fortiori} $\Phi'$-exact, admits a $\Psi$-exact
left Kan extension along it.
\end{proof}
\appendix
\section{Localisations of locally presentable categories}\label{sec:lexrefl}
The purpose of this appendix is to prove the following technical result, needed
for the arguments of Proposition~\ref{prop:PPhilexreflective},
Proposition~\ref{prop:setsingletonexact2} and
Theorem~\ref{thm:relativecompletion} above. In its statement, and throughout
this section, subcategory will always mean full, replete subcategory.
\begin{Prop}\label{prop:lexrefl}
If $\C$ is a locally presentable category in which finite limits commute with
filtered colimits, then the lex-reflective subcategories of $\C$ form a small,
complete lattice, in which infima are given by intersection.
\end{Prop}
As mentioned above, this result was proved for the unenriched case
in~\cite{Borceux1987On-locales}, appearing there as Theorem 6.8. We now recall
this proof, indicating along the way how it should be adapted to the enriched
setting in which we are working. The first step is to show:

\begin{Prop}\label{prop:lexreflaccessible}Any left exact reflector on a locally presentable
category preserves $\kappa$-filtered colimits for some regular cardinal
$\kappa$; equivalently, any localisation of a locally presentable category is
locally presentable.
\end{Prop}
\begin{proof}
Let $\E$ be lex-reflective in the locally $\lambda$-presentable $\C$; it
suffices to show that $\E$ is closed in $\C$ under $\kappa$-filtered colimits
for some $\kappa$, which is equally well to show that $\E_0$ is closed in
$\C_0$ under $\kappa$-filtered colimits, where $\E_0$ and $\C_0$ denote the underlying ordinary categories of $\E$ and $\C$. Now by Proposition 7.5
of~\cite{Kelly1982Structures}, $\C_0$ is locally $\lambda$-presentable since
$\C$ is; moreover, as finite conical limits are also finite weighted limits,
$\E_0$ is lex-reflective in $\C_0$. Thus it is enough to prove the result in
the unenriched case, and this is done in~\cite[Proposition
6.7]{Borceux1987On-locales}. We now briefly recall the argument.

Let $L_0 \colon \C_0 \to \E_0$ be the left exact reflection of $\C_0$ into
$\E_0$. Since $L_0$ preserves kernel-pairs, a standard result identifies $\E_0$
as the subcategory of $\C_0$ orthogonal to the class $\Sigma$ of all
monomorphisms inverted by $L_0$; see~\cite[Lemma
A4.3.6]{Johnstone2002Sketches}, for instance.
Now let $\Sigma_\lambda$ comprise those maps of $\Sigma$ whose codomain is
$\lambda$-presentable. Note that $\Sigma_\lambda$ is essentially-small---since
the $\lambda$-presentable objects span an essentially-small subcategory, and
$\C_0$ is well-powered---and so we can find a $\kappa$ bounding the rank of the
domains and codomains of the morphisms in it. Thus the subcategory orthogonal
to $\Sigma_\lambda$ is closed under $\kappa$-filtered colimits in $\C_0$ and we
will be done if we can show this subcategory is in fact $\E_0$. By the above,
this is equally well to show that any object orthogonal to $\Sigma_\lambda$ is
also orthogonal to each $m \colon A \to B$ in $\Sigma$. Given such an $m$, we
may, since $\C_0$ is locally $\lambda$-presentable, write $B$ as a
$\lambda$-filtered colimit of $\lambda$-presentables, $(q_i \colon X_i \to B
\mid i \in \I)$. Since $L_0$ preserves pullbacks, each $m_i \defeq q_i^\ast(m)$
is inverted by $L_0$ and so lies in $\Sigma_\lambda$. But since
$\lambda$-filtered colimits commute with $\lambda$-small limits in $\C_0$, they
are in particular stable under pullback, and so $m$ is the colimit in
$\C_0^\mathbf 2$ of the $m_i$'s; thus any object orthogonal to $\Sigma_\lambda$
is orthogonal to $m$, which concludes the proof.
\end{proof}

\begin{Cor}\label{cor:smalllexrefl}
Any locally presentable category has only a small set of lex-reflective
subcategories.
\end{Cor}
\begin{proof}
If $\E$ is lex-reflective in the locally $\lambda$-presentable $\C$, then as
above, $\E_0$ is lex-reflective in $\C_0$; and $\E_0$ clearly determines $\E$.
So it is enough to show that $\C_0$ has only a small set of lex-reflective
subcategories. The preceding proof shows that a left exact reflector on $\C_0$
is determined by the (replete) class of monomorphisms with
$\lambda$-presentable codomain that it inverts. But the class of all
monomorphisms with $\lambda$-presentable codomain is essentially-small, and so
has only a small set of replete subclasses.
\end{proof}
Given this, Proposition~\ref{prop:lexrefl} will now follow if we can prove the
following result; for the unenriched case, this is~\cite[Theorem
5.3]{Borceux1987On-locales}, and the argument given there carries over unchanged
to the $\V$-categorical setting.
\begin{Prop}
If finite limits commute with filtered colimits in the locally presentable
$\C$, then any small intersection of lex-reflective subcategories of $\C$ is
again lex-reflective.
\end{Prop}
\begin{proof}
We prove the result first for small, directed intersections, and then for
finite ones; this is clearly sufficient.
For the first of these, let $(\A_i \mid i \in \I)$ be a small, directed diagram
of lex-reflective subcategories of $\C$; we must show that $\A = \bigcap \A_i$
is also lex-reflective. Let each $\A_i$ have the reflector $\lambda_i \colon 1
\to L_i$, and let $\lambda \colon 1 \to L$ be the colimit of the directed
diagram formed by these reflectors. Since finite limits commute with filtered
colimits, and each $L_i$ is left exact, $L$ is too. It is moreover accessible,
since each $L_i$ is by Proposition~\ref{prop:lexreflaccessible}. Now, since
each $(L_i, \lambda_i)$ is a \emph{well-pointed} endofunctor---in the sense
that $L_i \lambda_i = \lambda_i L_i$---it follows from~\cite[Proposition
9.1]{Kelly1980A-unified} that $(L,\lambda)$ is too, and that an object $X$ lies
in $\A$ if and only if $\lambda_X \colon X \to LX$ is invertible, if and only
if it is orthogonal to each component of $\lambda$. Writing $\cat{On}$ for the
(large) poset of small ordinals, we now define a transfinite sequence
$L_{(\thg)} \colon \cat{On} \to [\C,\C]$ by
\begin{equation*}
    L_0 = 1_\C\text, \quad L_{\alpha + 1} = LX_\alpha\text, \quad X_\gamma = \colim_{\alpha < \gamma} X_\alpha\ \text,
\end{equation*}
with transition maps being given by $\lambda$ at successor stages, and by the
colimit injections at limit stages. Each $L_\alpha$ is left exact, since $L$
is, and since the colimits taken at limit stages are filtered. Moreover, any $X
\in \A$, being orthogonal to each component of $\lambda$, is also orthogonal to
each $L_0 \to L_\alpha$, so that if for some ordinal $\kappa$, the transition
map $\lambda L_\kappa \colon L_\kappa \to LL_\kappa$ is invertible, then
$L_\kappa$ will land in $\A$ and so provide the desired left exact reflection.
But since $L$ is accessible, it preserves $\kappa$-filtered colimits for some
$\kappa$; and thus, since $L$ is well-pointed, we deduce as in~\cite[Remark
6.3]{Kelly1980A-unified} that $\lambda L_\kappa  \colon L_\kappa \to LL_\kappa$
is invertible, so that $L_\kappa$ is the required left exact reflector into
$\A$.

This proves closure under directed intersections; as for finite ones, it
suffices to consider a binary intersection, and now the argument is almost
identical. Given $\Ss$ and $\T$ lex-reflective subcategories of $\C$,
corresponding to reflectors $\sigma \colon 1 \to S$ and $\tau \colon 1 \to T$,
we form the pointed endofunctor $\sigma \tau \colon 1 \to ST$, which is left
exact, well-pointed, and accessible, since $S$ and $T$ are. Moreover,
\cite[Proposition 4.3]{Borceux1987On-locales} shows that $X \in \Ss \cap \T$ if
and only if $\sigma \tau_X \colon X \to STX$ is invertible; whereupon the same
transfinite construction as before proves $\Ss \cap \T$ to be lex-reflective in
$\C$, as required.
\end{proof}


\bibliographystyle{acm}

\bibliography{bibdata}

\end{document}